\documentclass[oneside, a4paper,11pt,reqno]{amsart}
\usepackage{bbm}
\usepackage[T1]{fontenc}
\usepackage{amssymb}

\usepackage{verbatim}
\usepackage{amsmath}

\usepackage{amsfonts}
\usepackage{color}
\definecolor{Red}{rgb}{1,0,0}

\newtheorem{theorem}{Theorem}
\newtheorem{proposition}[theorem]{Proposition}
\newtheorem{cor}[theorem]{Corollary}
\newtheorem{lemma}[theorem]{Lemma}

\def \E {\mathbb{E}}
\def \P {\mathbb{P}}
\def \N  {\mathbb{N}}
\def \R  {\mathbb{R}}

\def \Z  {\mathbb{Z}}

\title[Velocity estimates for RWRE]{
	Velocity estimates for symmetric random walks 
	at  low ballistic disorder}

\author{Cl\'ement Laurent, Alejandro F. Ram\'\i rez, Christophe Sabot and Santiago Saglietti}

\email{clementelaurente@gmail.com, aramirez@mat.puc.cl, \newline \ \ \ \ \ sabot@math.univ-lyon1.fr, sasaglietti@mat.puc.cl}
\address{Institut Stanislas Cannes,
Facultad de Matem\'aticas, Pontificia Universidad Cat\'olica de Chile,
 Insitut Camille Jordan, Universit\'e de Lyon 1 \and
Facultad de Matem\'aticas, Pontificia Universidad Cat\'olica de Chile}

\thanks{Alejandro Ram\'\i rez and Santiago Saglietti have been
partially supported by Iniciativa Cient\'\i fica Milenio NC120062
and by Fondo Nacional de Desarrollo Cient\'\i fico
y Tecnol\'ogico grant 1141094. Cl\'ement Laurent has been 
partially supported by Fondo Nacional de Desarrollo Cient\'\i fico y
Tecnol\'ogico postdoctoral grant 3130353. Alejandro Ramirez and Christophe Sabot have been partially
suported by MathAmsud project ``Large scale behavior of stochastic systems''}

\date{\today}

\begin{document}

\begin{abstract}
	We derive asymptotic estimates for the velocity
	of random walks in random environments which are perturbations of the simple symmetric random walk but have a small local drift in a given direction. Our estimates 
complement previous results presented by Sznitman in \cite{Sz03} and are in the spirit of  expansions obtained by Sabot in \cite{Sa04}.
\end{abstract}
	
	\maketitle
	
\section{Introduction and Main Results}
	
	The mathematical derivation of explicit
	formulas for fundamental quantities of the model of random walk
	in a random environment is a challenging problem.
	For quantities like the velocity, the variance
	or the invariant measure of the environment seen from the random walk,
	few  results exist (see for example the review \cite{ST16} for the case
of Dirichlet environments, \cite{DR14} for one-dimensional
computations and also
	\cite{Sa04,CR16} for multidimensional expansions). In \cite{Sa04}, Sabot derived an asymptotic
	expansion for the velocity of the random walk at low disorder under
	the condition that the local drift of the perturbed
	random walk is linear in the perturbation parameter.
	As a corollary one can deduce that, in
	the case of perturbations of the simple symmetric random walk,
	the velocity is equal to the local drift with an error which
	is cubic in the perturbation parameter. In this article we
	explore up to which extent this expansion can be generalized
	to perturbations which are not necessarily linear in the perturbation
	parameter and we exhibit connections with previous results
	of Sznitman about ballistic behavior \cite{Sz03}.
	
	Fix an integer $d \geq 2$ and for $x=(x_1,\dots,x_d) \in \Z^d$ let $|x|:=|x_1|+\dots+|x_d|$ denote its $l^1$-norm. Let $V:=\{x\in\mathbb Z^d: |x|_1=1\}$ be the set of canonical vectors in $\Z^d$ and $\mathcal P$ denote the set of all probability vectors $\vec{p}=(p(e))_{e \in V}$ on $V$, i.e. such that $p(e)\geq 0$ for \mbox{all $e \in V$ and 
		also $\sum_{e\in V}p(e)=1$.} 
	Furthermore, let us consider the product space $\Omega:=\mathcal P^{\mathbb Z^d}$ 
	endowed with its Borel $\sigma$-algebra $\mathcal{B}(\Omega)$. We call any $\omega=(\omega(x))_{x \in \Z^d} \in \Omega$ an \textit{environment}. Notice
 that, for each $x\in\mathbb Z^d$, $\omega(x)$ is a probability vector on $V$, whose components we will denote by $\omega(x,e)$ for $e \in V$, i.e. $\omega(x)=(\omega(x,e))_{e \in V}$. 
	The {\it random walk in the environment} $\omega$
	starting from $x\in\Z^d$ is then defined as the Markov chain $(X_n)_{n \in \N_0}$ with
	state space $\Z^d$ which starts from $x$ and is given by the transition probabilities
	$$
	P_{x,\omega}(X_{n+1}=y+e|X_n=y)=\omega(y,e),
	$$
	for all $y\in\mathbb Z^d$ and $e\in V$. We will denote its law  by $P_{x,\omega}$.
	We assume throughout that the space of environments $\Omega$ is
	endowed with a probability measure $\P$, called the \mbox{{\it environmental law}.}
	We will call $P_{x,\omega}$  the {\it quenched law} of the random walk, and also refer to the semi-direct product \mbox{$P_x:= \P\otimes P_{x,\omega}$} defined on
	$\Omega\times{\mathbb Z}^{\mathbb N}$
	as the {\it averaged} or  {\it annealed law} of the random walk. In general, we will call the sequence $(X_n)_{n \in \N_0}$ under the annealed law a \textit{random walk in a random environment} (RWRE) with \mbox{environmental law $\P$.} Throughout the sequel, we will always assume that the random vectors $(\omega(x))_{x \in \Z^d}$ are \mbox{i.i.d. under $\P$.} Furthermore, we shall also assume that $\P$ is \textit{uniformly elliptic}, i.e. that there exits a constant $\kappa>0$ such that for all $x\in\mathbb Z^d$ and
	$e\in V$ one has
	$$
	\P(\omega(x,e)\ge\kappa)=1.
	$$
	
	Given $l\in\mathbb S^{d-1}$, we will say that our random walk $(X_n)_{n \in \N_0}$ is \textit{transient in direction
		$l$} 
	if
	$$
	\lim_{n\to\infty}X_n\cdot l=+\infty\quad P_0-a.s.,
	$$
	and say that it is {\it ballistic in direction $l$} if it satisfies the stronger condition
	$$
	\liminf_{n\to\infty}\frac{X_n\cdot l}{n}>0\qquad P_0-a.s.
	$$
	Any random walk which is ballistic with respect to some direction $l$ satisfies \mbox{a law of large numbers} (see \cite{DR14} for a proof of this fact), i.e. there exists a deterministic vector $\vec{v} \in \R^d$ with $\vec{v} \cdot l > 0$ such that
	$$
	\lim_{n \rightarrow +\infty} \frac{X_n}{n} = \vec{v} \qquad P_0-a.s..
	$$ This vector $\vec{v}$ is known as the \textit{velocity} of the random walk. 
	
	Throughout the following we will fix a certain direction, say $e_1:=(1,0,\dots,0) \in \mathbb{S}^{d-1}$ for example, and study transience/ballisticity only in this fixed direction. Thus, whenever we speak of transience or ballisticity of the RWRE it will be understood that it is with respect to this given direction $e_1$. However, we point out that all of our results can be adapted and still hold for any other direction. 
	
	For our main results, we will consider environmental laws $\P$ which are small
	perturbations of the simple symmetric random walk. More precisely, we will work with environmental laws $\P$ supported on the subset $\Omega_\epsilon \subseteq \Omega$ for $\epsilon > 0$ sufficiently small, where 
	\begin{equation}
	\label{epsilon-condition}
	\Omega_\epsilon:=
	\left\{\omega\in\Omega:\left|\omega(x,e)-\frac{1}{2d}\right|\le
	\frac{\epsilon}{4d} \,\text{ for all }x \in \Z^d \text{ and }e \in V\right\}.
	\end{equation}
	Notice that if $\P$ is supported on $\Omega_\epsilon$ for some $\epsilon \le 1$ then it is uniformly elliptic with constant
	\begin{equation}
	\label{eq:kappa}
	\kappa=\frac{1}{4d}.
	\end{equation} Since we wish to focus on RWREs for which there is ballisticity in direction $e_1$, it will be necessary to impose some further conditions on the environmental law $\P$. Indeed, 
 if for each $x \in \Z^d$ we define the {\it local drift of the RWRE at site $x$} as the random vector
	$$
	\vec{d}(x):=\sum_{e\in V}\omega(x,e)e
	$$ 
then, in order for the walk to be ballistic in direction $e_1$, one could expect that it is enough to have $\lambda:= \E(\vec{d}(0)) \cdot e_1 > 0$, where $\E$ here denotes the expectation with respect to the law $\P$ (notice that all local drift vectors $(\vec{d}(x))_{x \in \Z^d}$ are i.i.d. so that it suffices to consider only the local drift at $0$). However, as shown in \cite{BSZ03}, there are examples of environments for which there exists a direction in which the expectation of the local drift is positive but the velocity of the corresponding RWRE is negative. Therefore, we will need to impose stronger conditions on the local drift \mbox{to have ballisticity,} specifying exactly how small we allow $\lambda$ to be. In the sequel, we will consider two different conditions, the first of which is \textit{quadratic local drift condition}.
	
	\medskip
	\noindent \textbf{Quadratic local drift condition} (QLD). Given $\epsilon \in (0,1)$, we say that the environmental law $\P$ satisfies the quadratic local drift condition (QLD)$_{\epsilon}$ if $\P(\Omega_\epsilon)=1$ and, furthermore, 
	$$
	\lambda:=\E(\vec{d}(0)) \cdot e_1 \geq \epsilon^2.
	$$
	
	\medskip
	
	Our second condition, the \textit{local drift condition}, is weaker for dimensions $d \geq 3$.
	
	\medskip
	\noindent {\bf Local drift condition} (LD). Given $\eta,\epsilon \in (0,1)$, we say that
	an environmental law $\mathbb P$ satisfies 
	the local drift  condition (LD)$_{\eta,\epsilon}$ if  $\P(\Omega_\epsilon)=1$ and, furthermore, 
	
	\begin{equation}
	\label{lambda-definition}
	\lambda:=\E(\vec{d}(0))\cdot e_1 \ge
	\epsilon^{\alpha(d)-\eta},
	\end{equation}
	where
	\begin{equation}
	\label{epsilond}
	\alpha(d):=
	\begin{cases}
	2&{\rm if}\ d=2\\
	2.5&{\rm if}\ d=3\\
	3&{\rm if}\ d\ge 4.
	\end{cases}
	\end{equation}
	
	\medskip
	Observe that for $d=2$ and any $\epsilon \in (0,1)$ condition (LD)$_{\eta,\epsilon}$ implies (QLD)$_{\epsilon}$ for all $\eta \in (0,1)$, whereas if $d \geq 3$ and $\eta \in (0,\frac{1}{2})$ it is the other way round, (QLD)$_{\epsilon}$ implies (LD)$_{\eta,\epsilon}$. It is known that for every $\eta \in (0,1)$ there exists $\epsilon_0=\epsilon(d,\eta) > 0$ such that any RWRE with an environmental law $\P$ satisfying (LD)$_{\eta,\epsilon}$ for some $\epsilon \in (0,\epsilon_0)$ is ballistic. Indeed, for $d\ge 3$
	this was proved by Sznitman in \cite{Sz03} whereas the case $d=2$ was shown in \cite{R16} (and is also a consequence of Theorem \ref{theo:0} below). Therefore, any RWRE with an environmental law $\P$ which satisfies (LD)$_{\eta,\epsilon}$ for $\epsilon$ sufficiently small is such
	that $P_0$-a.s. the limit
	$$
	\vec{v}:=\lim_{n\to\infty}\frac{X_n}{n}
	$$
	exists and is different from $0$. Our first result is then the following.
	
\medskip
	\begin{theorem}
		\label{theorem1} Given any $\eta\in (0,1)$ and $\delta \in (0,\eta)$
		there exists some $\epsilon_0=\epsilon_0(d,\eta,\delta)\in (0,1)$ \mbox{such that}, for every $\epsilon \in (0,\epsilon_0)$ and any environmental law satisfying (LD)$_{\eta,\epsilon}$, the associated RWRE is ballistic with
		a velocity $\vec{v}$
		which verifies
		\begin{equation}
		\label{vel-expan}
		0 < \vec{v}\cdot e_1\le \lambda+ c_0\epsilon^{\alpha(d)-\delta}
		\end{equation} for some constant $c_0=c_0(d,\eta,\delta) > 0$. We abbreviate \eqref{vel-expan} by writing $0 < \vec{v} \cdot e_1 \leq \lambda + O_{d,\eta,\delta}(\epsilon^{\alpha(d)-\delta})$.
	\end{theorem}
	
\medskip
	
	Our second result is concerned with RWREs with an environmental law satisfying (QLD).

\medskip
	\begin{theorem}\label{theo:0}
		There exists $\epsilon_0 \in (0,1)$ depending only on the dimension $d$ such that for all $\epsilon \in (0,\epsilon_0)$ and any environmental law satisfying (QLD)$_\epsilon$, the associated RWRE is ballistic with a velocity $\vec{v}$ which verifies 
		$$
		|\vec{v} \cdot e_1 - \lambda| \leq \frac{\epsilon^2}{d}.
		$$
	\end{theorem}
	
\medskip

	 Combining both results we immediately obtain the following corollary. 

\medskip
	
	\begin{cor}\label{cor:main3} Given $\delta \in (0,1)$ there exists some $\epsilon_0=\epsilon_0(d,\delta)\in (0,1)$ such that, for all $\epsilon \in (0,\epsilon_0)$ and any environmental law satisfying (QLD)$_{\epsilon}$, the associated RWRE is ballistic with
		a velocity $\vec{v}$
		which verifies
		$$
		\lambda - \frac{\epsilon^2}{d} \leq \vec{v} \cdot e_1 \leq \lambda + O_{d,\delta}(\epsilon^{\alpha(d)-\delta}).
		$$
	\end{cor}

\medskip
	
	Observe that for dimension $d=2$ all the information given by Theorem \ref{theorem1} and Corollary \ref{cor:main3} is already contained in \mbox{Theorem \ref{theo:0},} whereas this is not so for dimensions $d \geq 3$. To 
 understand better the meaning of our results, let us give some background. First, 
	for $x \in \Z^d$ and $e \in V$ let us rewrite our weights $\omega(x,e)$ as
	\begin{equation}\label{eq:omega}
	\omega(x,e)=\frac{1}{2d} + \epsilon \xi_\epsilon(x,e),
	\end{equation} where
	$$
	\xi_\epsilon(x,e):=\frac{1}{\epsilon}\left(\omega(x,e)-\frac{1}{2d}\right).
	$$ Notice that if $\P(\Omega_\epsilon)=1$ then $\P$-almost surely we have $|\xi_\epsilon(x,e)|\leq \frac{1}{4d}$ for all $x \in \Z^d$ and $e \in V$. 
In \cite{Sa04}, Sabot considers a fixed environment $p_0 \in \Omega$ together with an i.i.d. sequence of bounded random vectors $\xi=(\xi(x))_{x \in \Z^d} \subseteq [-1,1]^V$ where each $\xi(x)=(\xi(x,e))_{e \in V}$ satisfies $\sum_{e\in V}\xi(x,e)=0$. Then, he defines for
	each $\epsilon>0$ the
	random environment $\omega$ on any $x\in\mathbb Z^d$ and $e\in V$ as
	$$
	\omega(x,e):=p_0(e)+\epsilon\xi(x,e).
	$$ In the notation of \eqref{eq:omega}, this
	corresponds to choosing $p_0(e)=\frac{1}{2d}$ and  $\xi_\epsilon(x,e):=\xi(x,e)$
	 not depending on
	$\epsilon$. 
	Under the assumption that the local drift associated to
	this RWRE does not vanish, it satisfies Kalikow's
	condition \cite{K81} and thus it has a non-zero velocity $\vec{v}$.
	Sabot then proves that this velocity
	satisfies the following expansion: for any small
	$\delta>0$ there exists some $\epsilon_0=\epsilon_0(d,\delta)>0$ such that for any $\epsilon \in (0,\epsilon_0)$
	one has that
	
\begin{equation}
\label{sabotexp2}
	\vec{v}= \vec{d}_0+
	\epsilon \vec{d}_1+\epsilon^2 \vec{d}_2+O_{d,\delta}\left(\epsilon^{3-\delta}\right),
\end{equation}
	where
	$$
	\vec{d}_0:=\sum_{e\in V}p_0(e)e,\qquad \vec{d}_1:=\sum_{e\in V}\mathbb E[\xi(0,e)]e,
	$$
	and
	$$
	\vec{d}_2:=\sum_{e \in V}\left(\sum_{e'\in V}C_{e,e'}J_{e'}\right) e,
	$$
	with
	$$
	C_{e,e'}:=\text{Cov}(\xi(0,e),\xi(0,e'))\qquad{\rm and}\qquad
	J_e:=g_{p_0}(e,0)-g_{p_0}(0,0).
	$$
	Here $g_{p_0}(x,y)$ denotes the Green's function of a random
	walk with jump kernel $p_0$. It \mbox{turns
		out that} for the particular case in which $p_0$
	is the jump kernel of a simple symmetric random walk (which is
	the choice we make in this article), we have that $\vec{d}_0=0$
	and also $\vec{d}_2=0$. In particular, for this case we have $\lambda= \epsilon \vec{d}_1 \cdot e_1=O(\epsilon)$ and 
	\begin{equation}
\label{sabotexp}
	\vec{v} \cdot e_1 = \lambda + O_{d,\delta}\left(\epsilon^{3-\delta}\right).
	\end{equation}
	Even though this expansion was only shown valid in the regime $\lambda = O(\epsilon)$, from it one can guess that, at least at a formal level, the random walk should be ballistic whenever $\lambda\ge \epsilon^{3-\eta}$ for\mbox{ any $\eta>\delta$.} 
	This was established previously by Sznitman from \cite{Sz03} for dimensions $d\ge 4$, but remains open for dimensions $d=2$ and $d=3$. In this context, our results show that
	under the drift condition (LD), which is always weaker than the $\lambda=O(\epsilon)$ assumption in \cite{Sa04}, for $d=2$ the random walk is indeed ballistic and the expansion \eqref{sabotexp} is still valid up to the second order (Theorem \ref{theo:0}), whereas for $d \geq 3$ we show that at least an upper estimate 
compatible with the right-hand side of \eqref{sabotexp}
 holds for the velocity (Theorem \ref{theorem1}).
	
	The proof of Theorem \ref{theorem1} is rather different from
the proof of the velocity expansion (\ref{sabotexp2}) of \cite{Sa04}, and is based on a mixture of
	renormalization methods together with Green's functions estimates,
	inspired in methods presented in \cite{Sz03,BDR14}.
	As a first step, one shows that the \mbox{averaged velocity} of the
	random walk at distances of order $\epsilon^{-4}$ is precisely
	equal to the average of the \mbox{local drift} with an error of order
	$\epsilon^{\alpha(d)-\delta}$. To do this, essentially we show
	that a right approximation for the behavior of the random walk
	at distances $\epsilon^{-1}$ is that of a simple symmetric random walk,
	so that one has to find a good estimate for the probability to move
	to the left or to the right of a rescaled random walk moving on a
	grid of size $\epsilon^{-1}$. This last estimate is obtained
	through a careful approximation of the Green's function
	of the random walk, which involves comparing it with its average
by using a martingale method. This is a crucial step which explains
	the fact that one loses precision in the error of the velocity in
	dimensions $d=2$ and $d=3$ compared with $d\ge 4$. As a final
	result of these computations, we obtain that the polynomial
	condition of \cite{BDR14} holds. In the second step, we
	use a renormalization method to derive the upper bound for the velocity, using the polynomial condition proved in the first step as a seed estimate.
	The proof of Theorem \ref{theo:0} is somewhat simpler, and is based
	on a generalization of Kalikow's formula proved in \cite{Sa04} and a careful application of Kalikow's criteria for ballisticity.
	
	The article is organized as follows. In Section 2 we introduce the general notation and
	establish some preliminary facts about the RWRE model, including some useful
	Green's function estimates. In Section 3, we prove Theorem \ref{theo:0}.
	In Section 4, we obtain the velocity estimates for distances of
	order $\epsilon^{-4}$ which is the first step in the proof
	of Theorem \ref{theorem1}. Finally, in Section 5 we finish
	the proof of Theorem \ref{theorem1} through the renormalization argument
	described above.
	
\section{Preliminaries}\label{section2}

In this section we introduce the general notation to be used throughout the article and also review some basic facts about RWREs which we shall need later.

\subsection{General notation}\label{sec:GN}
Given any subset $A\subset\mathbb Z^d$, we define
its (outer) boundary as
$$
\partial A:=\{x \in \Z^d - A : |x-y|=1 \text{ for some }y \in A\}.
$$
Also, we define the first exit
time of the random walk from $A$ as
$$
T_A:=\inf\{n\ge 0: X_n\notin A\}.
$$ In the particular case in which $A=\{b\} \times \Z^{d-1}$ for some $b \in \Z$, we will write $T_b$ instead of $T_A$, i.e.
$$
T_b:=\inf \{ n \geq 0 : X_n \cdot e_1 = b\}.
$$
Throughout the rest of this paper $\epsilon>0$ will be 
treated as a fixed variable. Also, we will denote generic constants by $c_1,c_2,\dots$.
However, whenever we wish to highlight the dependence of any of
these constants on the dimension $d$ or on $\eta$,
we will write for example $c_1(d)$ or $c_1(\eta,d)$ \mbox{instead
	of $c_1$.}
Furthermore, for the sequel we will fix a constant $\theta \in (0,1)$ to be determined later and define
\begin{equation}
\label{defL}
L:=2[\theta \epsilon^{-1}]
\end{equation}
where $[\cdot]$ denotes the (lower) integer part and also
\begin{equation}
\label{defN}
N:=L^3,
\end{equation}
which will be used as length quantifiers. 
In the sequel we will often work with slabs and \mbox{boxes in $\Z^d$,} which we introduce now. For each $M \in \N$, $x \in \Z^d$ and $l\in\mathbb S^{d-1}$ we define the slab 
\begin{equation}\label{eq:slab}
U_{l,M}(x):=\left\{y\in\mathbb Z^d: -M \leq (y-x)\cdot l < M
\right\}.
\end{equation} 
Whenever $l=e_1$ we will suppress $l$ from the notation and write $U_M(x)$ instead. Similarly, whenever $x=0$ we shall write $U_M$ instead of $U_M(0)$ and abbreviate $U_L(0)$ simply as $U$ for $L$ as defined \eqref{defL}.
Also, for each $M \in \N$ and $x\in\mathbb Z^d$, we define the box
\begin{equation}
\label{beeme}
B_{M}(x):=\left\{y\in\mathbb Z^d: - \frac{M}{2} < (y-x) \cdot e_1 < M \text{ and }
|(y-x)\cdot e_i| < 25M^3 \text{ for } 2\le i\le d\right\}
\end{equation}
together with its {\it frontal side}
$$
\partial_+B_{M}(x):=\left\{y\in \partial B_{M,M'}(x): (y-x)\cdot e_1 \geq M\right\},
$$ its \textit{back side}
$$
\partial_- B_M(x):= \left\{y\in \partial B_{M,M'}(x): (y-x)\cdot e_1 \leq -\frac{M}{2}\right\},
$$ its \textit{lateral side}
$$
\partial_l B_M(x):= \left\{y\in \partial B_{M,M'}(x): |(y-x)\cdot e_i| \geq 25M^3 \text{ for some }2 \leq i \leq d\right\},
$$ and, finally, its \textit{middle-frontal part}
$$
B^*_M(x):= \left\{ y \in B_{M}(x) : \frac{M}{2} \leq (y-x) \cdot e_1 < M\,,\,|(y-x) \cdot e_i| < M^3 \text{ for }2 \leq i \leq d\right\}
$$ together with its corresponding \textit{back side}
$$
\partial_- B^*_M(x):=\left\{ y \in B^*_M(x) : (y-x)\cdot e_1 = \frac{M}{2}\right\}.
$$
As in the case of slabs, we will use the simplified notation 
$B_{M}:=B_{M}(0)$ and also $\partial_i B_M:=\partial_i B_M(0)$ for $i=+,-,l$, with the analogous simplifications for $B_M^*(0)$ and its back side. 

\subsection{Ballisticity conditions} \label{sec:bal}
For the development of the proof of our results, it will be important to recall a few ballisticity conditions, namely, Sznitman's $(T)$ and $(T')$ conditions introduced
in \cite{Sz01,Sz02} and also the
polynomial condition presented in \cite{BDR14}. We do this now, considering only ballisticity in direction $e_1$ for simplicity.

\medskip
\noindent \textbf{Conditions $(T)$ and $(T')$}. Given $\gamma\in (0,1]$ we say that condition $(T)_\gamma$ is satisfied (in direction $e_1$) if there exists a neighborhood $V$ of $e_1$ in $\mathbb S^{d-1}$ such that for every $l'\in V$
one has that
\begin{equation}\label{eq:condtgamma}
\limsup_{M\to+\infty}\frac{1}{M^\gamma}\log P_0\left(X_{T_{U_{l',M}}}\cdot l'<0\right)<0.
\end{equation}
As a matter of fact, Sznitman originally introduced a condition $(T)_{\gamma}$ which is slightly different from the one presented here, involving an asymmetric version of the slab $U_{l',M}$ in \eqref{eq:condtgamma} and an additional parameter $b > 0$ which modulates the asymmetry of this slab. However, it is straightforward to check that Sznitman's original definition is equivalent to ours, so we omit it for simplicity.

Having defined the conditions $(T)_\gamma$ for all $\gamma \in (0,1]$, we will say that:
\begin{itemize}
	\item $(T)$ is satisfied (in direction $e_1$) if $(T)_1$ holds,
    \item $(T')$ is satisfied (in direction $e_1$) if $(T)_\gamma$ holds for all $\gamma\in (0,1)$. 
\end{itemize}
It is clear that $(T)$ implies $(T')$, although it is not yet known whether the other implication holds.

\medskip
\noindent \textbf{Condition $(P)_K$}. Given $K \in \N$ we say that the polynomial condition $(P)_K$ holds (in direction $e_1$) if for some $M \ge M_0$ one has that
$$
\sup_{x \in B^*_{M}}P_x\left(X_{T_{B_{M}}}\notin\partial_+ B_{M}\right)\le\frac{1}{M^K},
$$ where 
\begin{equation}\label{eq:defM0}
M_0:=\exp\left\{ 100 +4d (\log \kappa)^2\right\}
\end{equation} where $\kappa$ is the uniform ellipticity constant, which in our present case can be taken as $\kappa=\frac{1}{4d}$\mbox{, see \eqref{eq:kappa}.} It is well-known that both $(T')$ and $(P)_K$ imply ballisticity in direction $e_1$, see \cite{Sz02,BDR14}. Furthermore, in \cite{BDR14} it is shown that
$$
(P)_K \text{ holds for some }K \geq 15d+5 \Longleftrightarrow (T') \text{ holds } \Longleftrightarrow (T)_{\gamma} \text{ holds for some }\gamma \in (0,1).
$$

\subsection{Green's functions and operators}
Let us now introduce some notation we shall use related to the Green's functions of the RWRE and of the simple symmetric random walk (SSRW). 

Given a subset $B \subseteq \Z^d$, the Green's functions of the RWRE and SSRW killed upon exiting $B$ are respectively defined for $x,y \in B \cup \partial B$ as
$$
g_B(x,y,\omega):= E_{x,\omega}\left(\sum_{n=0}^{T_{B}} \mathbbm{1}_{\{X_n=y\}}\right)\hspace{1cm}\text{ and }\hspace{1cm}g_{0,B}(x,y):=g_{B}(x,y,\omega_0),
$$ where $\omega_0$ is the corresponding weight of the SSRW, given for all $x \in \Z^d$ and $e \in V$ by
$$
\omega_0(x,e)=\frac{1}{2d}.
$$ Furthermore, if $\omega \in \Omega$ is such that $E_{x,\omega}(T_B) < +\infty$ for all $x \in B$, we can define the corresponding Green's operator on $L^\infty(B)$ by the formula
$$
G_B[f](x,\omega):= \sum_{y \in B} g_B(x,y,\omega)f(y).
$$ Notice that $g_B$, and therefore also $G_B$, depends on $\omega$ only though its restriction $\omega|_B$ to $B$. Finally, it is straightforward to check that if $B$ is a slab as defined in \eqref{eq:slab} then both $g_B$ and $G_B$ are well-defined for all environments $\omega \in \Omega_\epsilon$ with $\epsilon \in (0,1)$.

\section{Proof of Theorem \ref{theo:0}}

The proof of Theorem \ref{theo:0} has several steps. We begin by establishing a law of large numbers for the sequence of hitting times $(T_n)_{n \in \N}$.

\subsection{Law of large numbers for hitting times}\label{sec:LLN}

We now show that, under the condition $(P)_K$, the sequence of hitting times $(T_n)_{n \in \N}$ satisfies a law of large numbers with the inverse of the velocity in direction $e_1$ as its limit.

\begin{proposition}
\label{prop-time} If $(P)_K$ is satisfied for some $K \geq 15d+5$ then $P_0$-a.s. we have that
\begin{equation}
\label{time-limitt}
\lim_{n\to\infty}\frac{E_0(T_n)}{n}
=\lim_{n\to\infty}\frac{T_n}{n}=\frac{1}{\vec{v}\cdot e_1} > 0,
\end{equation}
where $\vec{v}$ is the velocity of the corresponding RWRE.
\end{proposition} 

To prove Proposition \ref{prop-time}, we will
require the following lemma and its subsequent corollary.

\begin{lemma}
\label{slowdowns} If $(P)_K$ holds for some $K \ge 15d+5$ then there exists $c_1>0$ such that for each $n \in \N$ and all $a>\frac{1}{v\cdot e_1}$ one has that
\begin{equation}
\label{ub-time}
P_0\left(\frac{T_n}{n}\ge a\right)\le\frac{1}{c_1}
\exp\left\{-c_1\left(\left(\log\left(a-\frac{1}{\vec{v}\cdot e_1}\right)\right)^{\frac{2d-1}{2}}+
\left(\log(n)\right)^{\frac{2d-1}{2}}\right)\right\}.
\end{equation}
\end{lemma}

\begin{proof} 
By Berger, Drewitz and Ram\'irez \cite{BDR14},
we know that since $(P)_K$ holds for $K\ge 15d+5$,
necessarily $(T')$ must also hold. Now,
a careful examination of the proof of Theorem 3.4 in \cite{Sz02} shows that the upper bound in \eqref{ub-time} is satisfied.
\end{proof}

\begin{cor}
\label{uniform-integrability} If $(P)_K$ holds for some $K \geq 15d+5$ then $\left(\frac{T_n}{n}\right)_{n \in \N}$ is uniformly $P_0$-integrable.
\end{cor}

\begin{proof} Note that, by Lemma \ref{slowdowns}, for $K>\frac{1}{v\cdot e_1}$ and $n\ge 2$ we have that

$$
\int_{\{ \frac{T_n}{n} \ge K\}} \frac{T_n}{n}dP_0
\le \sum_{k=K}^\infty (k+1) P_0\left(\frac{T_n}{n}\ge k\right)
\le
\frac{1}{c_1}
\sum_{k=K}^\infty (k+1)
e^{-c_1\left(\log\left(k-\frac{1}{\vec{v}\cdot e_1}\right)\right)^{
\frac{2d-1}{2}}-
c_1\left(\log(2)\right)^{\frac{2d-1}{2}}}.
$$
From here it is clear that, since $d \geq 2$, we have
$$
\lim_{K\to\infty} \left[\sup_{n\ge 1}\int_{\{ \frac{T_n}{n} \geq K\}} \frac{T_n}{n}dP_0\right]=0
$$ which shows the uniform $P_0$-integrability.
\end{proof}

Let us now see how to obtain Proposition \ref{prop-time} from Corollary \ref{uniform-integrability}. Since $(P)_K$ holds for $K \geq 15d+5$, by Berger, Drewitz and Ram\'\i rez \cite{BDR14} we know that the position of the random walk satisfies a law of large numbers with a velocity $\vec{v}$ 
such that $\vec{v}\cdot e_1>0$. 
Now, note that for any $\varepsilon>0$ one has
\begin{align*}
P_0\left(\left|\frac{n}{T_n}-\vec{v}\cdot e_1\right|\ge\varepsilon\right)
&=
P_0\left(\left|\frac{X_{T_n}\cdot e_1}{T_n}-\vec{v}\cdot e_1\right|
\ge\varepsilon\right)
\\
& \le
\sum_{k=n}^\infty P_0\left(\left|\frac{X_{k}\cdot e_1}{k}-\vec{v}\cdot e_1\right|\ge\varepsilon\right)\\
& \le\sum_{k=n}^\infty e^{-C(\log k)^{\frac{2d-1}{2}}},
\end{align*}
where in the last inequality we have used the slowdown estimates
for RWREs satisfying $(T')$
proved by Sznitman in \cite{Sz02} (see also the improved result
of Berger in \cite{B12}). Hence, by Borel-Cantelli we
conclude that $P_0$-a.s.
$$
\lim_{n\to\infty}\frac{n}{T_n}=\vec{v}\cdot e_1, 
$$
from where the second equality of \eqref{time-limitt} immediately follows. The first one is now a direct consequence of the uniform integrability provided by Corollary \ref{uniform-integrability}.

\subsection{Introducing Kalikow's walk} Given a nonempty connected strict subset $B \subsetneq \Z^d$, for $x \in B$ we define \textit{Kalikow's walk} on $B$ (starting from $x$) as the random walk starting from $x$ which is killed upon exiting $B$ and has transition probabilities determined by the environment $\omega_B \in \mathcal{P}^B$ given by 
\begin{equation}
\label{eq:defkalb}
\omega_B^x(y,e):=\frac{ \E( g_B(x,y,\omega)\omega(y,e))}{\E(g_B(x,y,\omega))}.
\end{equation} It is straightforward to check that by the uniform ellipticity of $\P$ we have $0< \E(g_B(x,y,\omega)) <+\infty$ for all $y \in B$, so that the environment $\omega_B^x$ is well-defined. In accordance with our present notation, we will denote the law of Kalikow's walk on $B$ by $P_{x,\omega_B^x}$ and its Green's function by $g_B(x,\cdot,\omega_B^x)$. The importance of Kalikow's walk, named after S. Kalikow who originally introduced \mbox{it in \cite{K81},} lies in the following result which is a slight generalization of Kalikow's formula proved in
\cite{K81} and of the statement of it given in \cite{Sa04}.

\begin{proposition} \label{prop:kali} If $B \subsetneq \Z^d$ is connected then for any $x \in B$ with $P_{x,\omega_B^x}(T_B < +\infty) = 1$ we have
	\begin{equation}\label{eq:sabkal}
	\E(g_B(x,y)) = g_B(x,y,\omega_B^x)
	\end{equation} for all $y \in B \cup \partial B$.
\end{proposition}

\begin{proof} The proof is similar to that of \cite[Proposition 1]{Sa04}, but we include it here for completeness. First, let us observe that for any $\omega \in \Omega_\epsilon$ and $y \in B \cup \partial B$ we have by the Markov property 
	\begin{align*}
	g_B(x,y,\omega)&= E_{x,\omega}\left( \sum_{n=0}^{T_B} \mathbbm{1}_{\{X_n=y\}}\right)\\
	& = \sum_{n=0}^{\infty} P_{x,\omega}\left( X_n = y , T_B \geq n \right)\\
	& = \mathbbm{1}_{\{x\}}(y) + \sum_{n=1}^{\infty} \sum_{e \in V} P_{x,\omega}\left(X_{n-1} = y-e,  X_n = y, T_B > n - 1\right)\\
	& = \mathbbm{1}_{\{x\}}(y) + \sum_{e \in V} \sum_{n=1}^{\infty}  P_{x,\omega}\left(X_{n-1} = y-e, T_B > n - 1\right)\omega(y-e,e)\\
	& = \mathbbm{1}_{\{x\}}(y) + \sum_{e \in V} \mathbbm{1}_{B}(y-e) g_B(x,y-e,\omega)\omega(y-e,e),
	\end{align*} so that 
	$$
	\E(g_B(x,y))= \mathbbm{1}_{\{x\}}(y) + \sum_{e \in V\,:\,y-e \in B} \E(g_B(x,y-e))\omega_B^x(y-e,e).
	$$ Similarly, if for each $k \in \N_0$ we define 
	$$
	g_B^{(k)}(x,y,\omega_B^x):=E_{x,\omega_B^x}\left( \sum_{n=0}^{T_B \wedge k} \mathbbm{1}_{\{X_n=y\}}\right)
	$$ then by the same reasoning as above we obtain 
	\begin{equation}
	\label{eq:kali3}
	g_B^{(k+1)}(x,y,\omega_B^x)= \mathbbm{1}_{\{x\}}(y) + \sum_{e \in V\,:\,y-e \in B}g^{(k)}_B(x,y-e,\omega_B^x)\omega_B^x(y-e,e).
	\end{equation} In particular, we see that for all $k \in \N_0$
	$$
	\E(g_B(x,y)) - g_B^{(k+1)}(x,y,\omega_B^x) = \sum_{e \in V\,:\,y-e \in B}\left(\E(g_B(x,y-e)) - g_B^{(k)}(x,y-e,\omega_B^x)\right)\omega_B^x(y-e,e)
	$$ which, since $\omega_B^x$ is nonnegative and also $g_B^{(0)}(x,y,\omega_B^x) = \mathbbm{1}_{\{x\}}(y) \leq \E(g_B(x,y))$ for every $y \in B \cup \partial B$, by induction implies that $g_B^{(k)}(x,y,\omega_B^x) \leq \E(g_B(x,y))$ for all $k \in \N_0$. Therefore, by letting $k \rightarrow +\infty$ in this last inequality we obtain 
	\begin{equation}\label{eq:kali1}
	g_B(x,y,\omega_B^x) \leq \E(g_B(x,y))
	\end{equation}
	for all $y \in B \cup \partial B$. In particular, this implies that 
	\begin{equation}\label{eq:kali2}
	P_{x,\omega_B^x}(T_B < +\infty) = \sum_{y \in \partial B} g_B(x,y,\omega_B^x) \leq \sum_{y \in \partial B} \E(g_B(x,y)) = P_x(T_B < +\infty) \leq 1.
	\end{equation} Thus, if $P_{x,\omega_B^x}(T_B<+\infty)=1$ then both sums on \eqref{eq:kali2} are in fact equal which, together with \eqref{eq:kali1}, implies that 
	$$
	g_B(x,y,\omega_B^x) = \E(g_B(x,y))
	$$ for all $y \in \partial B$. Finally, to check that this equality also holds for every $y \in B$, we first notice that for any $y \in B \cup \partial B$ we have by \eqref{eq:kali3} that
	$$
	g_B(x,y,\omega_B^x) = \mathbbm{1}_{\{x\}}(y) + \sum_{e \in V\,:\,y-e \in B}g_B(x,y-e,\omega_B^x)\omega_B^x(y-e,e)
	$$ so that if $y \in B \cup \partial B$ is such that $\E(g_B(x,y))=g_B(x,y,\omega_B^x)$ then
	$$
	0 =  \sum_{e \in V\,:\,y-e \in B}º\left(\E(g_B(x,y-e))-g_B(x,y-e,\omega_B^x)\right)\omega_B^x(y-e,e).
	$$ Hence, by the nonnegativity of $\omega_B^x$ and \eqref{eq:kali1} we conclude that if $y \in B \cup \partial B$ is such that \eqref{eq:sabkal} holds then \eqref{eq:sabkal} also holds for all $z \in B$ of the form $z=y-e$ for some $e \in V$. Since we already have that \eqref{eq:sabkal} holds for all $y \in \partial B$ and $B$ is connected, by induction one can obtain \eqref{eq:sabkal} for all $y \in B$.
	\end{proof}  

As a consequence of this result, we obtain the following useful corollary,
which is the original formulation of Kalikow's formula \cite{K81}.

\begin{cor}\label{cor:kali} If $B \subsetneq \Z^d$ is connected then for any $x \in B$ such that $P_{x,\omega_B^x}(T_B < +\infty) = 1$ we have
	$$
	E_x(T_B)=E_{x,\omega_B^x}(T_B)
	$$ and 
	$$
	P_{x}(X_{T_B} = y) = P_{x,\omega_B^x}(X_{T_B}= y )
	$$ for all $y \in \partial B$.
\end{cor}

\begin{proof} This follows immediately from Proposition \ref{prop:kali} upon noticing that, by definition of $g_B$, we have on the one hand
	$$
	E_x(T_B)=\sum_{y \in B} \E(g_B(0,y)) = \sum_{y \in B} g_B(x,y,\omega_B^x) = E_{x,\omega_B^x}(T_B)
	$$ and, on the other hand, for any $y \in \partial B$ 
	$$
	P_{x}(X_{T_B} = y )= \E(g_B(x,y))=g_B(x,y,\omega_B^x)=P_{x,\omega_B^x}(X_{T_B}=y).
	$$
\end{proof}

Proposition \ref{prop-time} shows that in order to obtain bounds on $\vec{v}\cdot e_1$, the velocity in direction $e_1$, it might be useful to understand the behavior of the expectation $E_0(T_n)$ as $n$ tends to infinity, provided that the polynomial condition $(P)_K$ indeed holds for $K$ sufficiently large. As it turns out, Corollary \ref{cor:kali} will provide a way in which to verify the polynomial condition together with the desired bounds for $E_0(T_n)$ by means of studying the killing times of certain auxiliary Kalikow's walks. To this end, the following lemma will play an important role.

\begin{lemma}\label{prop:kalidrift}
	If given a connected subset $B \subsetneq \Z^d$ and $x \in B$ we define for each $y \in B$ the drift at $y$ of the Kalikow's walk on $B$ starting from $x$ as 
	$$
	\vec{d}_{B,x}(y) := \sum_{e \in V} \omega_B^x(y,e) e
	$$ where $\omega_B^x$ is the environment defined in \eqref{eq:defkalb}, then 
	$$
	\vec{d}_{B,x}(y)= \frac{\E\left( \frac{\vec{d}(y,\omega)}{\sum_{e \in V}\omega(x,e)f_{B,x}(y,y+e,\omega)}\right)}{\E\left( \frac{1}{\sum_{e \in V} \omega(x,x+e)f_{B,x}(y,y+e,\omega)}\right)}.
	$$ where $f_{B,x}$ is given by
	$$
	f_{B,x}(y,z,\omega):=\frac{P_{z,\omega}(T_B \leq H_y)}{P_{x,\omega}(H_y < T_B)}
	$$ and $H_y:=\inf\{n \in \N_0 : X_n = y\}$ denotes the hitting time of $y$.
\end{lemma}

\begin{proof} Observe that if for $y,z \in B \cup \partial B$ and $\omega \in \Omega$ we define 
	$$
	g(y,z,\omega):=P_{z,\omega}( H_y < T_B )
	$$ then by the strong Markov property we have for any $y \in B$
	$$
	\E(g_B(x,y,\omega)) = \E\left( E_{x,\omega}\left(\sum_{n=0}^{T_B} \mathbbm{1}_{\{X_n = y\}}\right)\right) = \E\left( g(y,x,\omega)E_{y,\omega}\left(\sum_{n=0}^{T_B} \mathbbm{1}_{\{X_n = y\}}\right)\right).
	$$
	Now, under the law $P_{y,\omega}$, the total number of times $n \in \N_0$ in which the random walk $X$ is at $y$ before exiting $B$ is a geometric random variable with success probability
	$$
	p:=\sum_{e \in V} \omega(y,y+e)(1-g(y,y+e,\omega)),
	$$ so that
	$$
	E_{y,\omega}\left(\sum_{n=0}^{T_B} \mathbbm{1}_{\{X_n = y\}}\right) = \frac{1}{\sum_{e \in V}\omega(y,y+e)(1-g(y,y+e,\omega))}.
	$$ It follows that
	$$
	\E(g_B(x,y,\omega))=\E\left( \frac{1}{\sum_{e \in V} \omega(y,y+e)f_{B,x}(y,y+e,\omega)}\right)
	$$ where $f_{B,x}$ is defined as
	$$
	f_{B,x}(y,z,\omega):=\frac{1-g(y,z,\omega)}{g(y,x,\omega)}=\frac{P_{z,\omega}(T_B \leq H_y)}{P_{x,\omega}(H_y < T_B)}.
	$$ By proceeding in the same manner, we also obtain
	$$
	\E(g_B(x,y,\omega)\omega(y,e))= \E\left( \frac{\omega(y,e)}{\sum_{e \in V} \omega(y,y+e)f_{B,x}(y,y+e,\omega)}\right), 
	$$ so that 
	$$
	\vec{d}_{B,x}(y)= \frac{\E\left( \frac{\vec{d}(y,\omega)}{\sum_{e \in V}\omega(y,y+e)f_{B,x}(y,y+e,\omega)}\right)}{\E\left( \frac{1}{\sum_{e \in V} \omega(y,y+e)f_{B,x}(y,y+e,\omega)}\right)}.
	$$
\end{proof}

As a consequence of Lemma \ref{prop:kalidrift}, we obtain the following key estimates on the drift of \mbox{Kalikow's walk.} 

\begin{proposition}\label{prop:kalibound} If $\P$ satisfies (QLD)$_{\epsilon}$ for some $\epsilon \in (0,1)$ then for any connected subset $B \subsetneq \Z^d$ and $x \in B$ we have
	$$
	\sup_{y \in B} \left[|\vec{d}_{B,x}
	(y) \cdot e_1 - \lambda|\right] \leq \frac{\epsilon^2}{d}.
	$$
\end{proposition} 

\begin{proof} First, let us decompose
	$$
	\E\left( \frac{\vec{d}(y,\omega) \cdot e_1}{\sum_{e \in V}\omega(y,e)f_{B,x}(y,y+e,\omega)}\right) = \E\left( \frac{(\vec{d}(y,\omega) \cdot e_1)_+-(\vec{d}(y,\omega) \cdot e_1)_-}{\sum_{e \in V}\omega(y,e)f_{B,x}(y,y+e,\omega)}\right).
	$$ Now, notice that since $\P(\Omega_\epsilon)=1$ we have
	$$
	\frac{1}{\sum_{e \in V}\omega(y,e)f_{B,x}(y,y+e,\omega)} \leq \frac{1}{\sum_{e \in V}\left(\frac{1}{2d} -\frac{\epsilon}{4d}\right)f_{B,x}(y,y+e,\omega)}\\
	\leq \frac{2d}{\sum_{e \in V}f_{B,x}(y,y+e,\omega)} \cdot \frac{1}{1-\frac{\epsilon}{2}}
	$$ and also 
	$$
	\frac{1}{\sum_{e \in V}\omega(y,e)f_{B,x}(y,y+e,\omega)} \geq \frac{2d}{\sum_{e \in V}f_{B,x}(y,y+e,\omega)} \cdot \frac{1}{1+\frac{\epsilon}{2}}.
	$$ In particular, we obtain that
	$$
	\E\left( \frac{(\vec{d}(y,\omega) \cdot e_1)_\pm}{\sum_{e \in V}\omega(y,e)f_{B,x}(y,y+e,\omega)}\right) \leq \frac{2d}{1-\frac{\epsilon}{2}} \cdot \E\left( \frac{1}{\sum_{e \in V}f_{B,x}(y,y+e,\omega)}\cdot (\vec{d}(y,\omega)\cdot e_1)_\pm\right)
	$$
	and 
	$$
	\E\left( \frac{(\vec{d}(y,\omega) \cdot e_1)_\pm}{\sum_{e \in V}\omega(y,e)f_{B,x}(y,y+e,\omega)}\right) \geq \frac{2d}{1+\frac{\epsilon}{2}} \cdot \E\left( \frac{1}{\sum_{e \in V}f_{B,x}(y,y+e,\omega)}\cdot (\vec{d}(y,\omega)\cdot e_1)_\pm\right)
	$$ Furthermore, since $\sum_{e \in V}f_{B,x}(y,y+e,\omega)$ is independent of $\omega(y)$ it follows that
	$$
	\E\left( \frac{1}{\sum_{e \in V}f_{B,x}(y,y+e,\omega)}\cdot (\vec{d}(y,\omega)\cdot e_1)_\pm\right) = \E\left( \frac{1}{\sum_{e \in V}f_{B,x}(y,y+e,\omega)}\right) \cdot \E\left( (\vec{d}(y,\omega)\cdot e_1)_\pm\right).
	$$ Finally, by combining this with the previous estimates, a straightforward calculation yields 
	$$
	\lambda - \frac{4}{4-\epsilon^2} \cdot \epsilon \cdot \E(|\vec{d}(y,\omega)\cdot e_1|) + \frac{2\epsilon^2}{4-\epsilon^2}\cdot \lambda \leq \vec{d}_{B,x}(y) \leq  \lambda + \frac{4}{4-\epsilon^2} \cdot \epsilon \cdot \E(|\vec{d}(y,\omega)\cdot e_1|) + \frac{2\epsilon^2}{4-\epsilon^2} \cdot \lambda 
	$$ Since $\P(\Omega_\epsilon)=1$ implies that $|\lambda|\leq \E(|\vec{d}(y,\omega)\cdot e_1|) \leq \frac{\epsilon}{2d}$, by recalling that $\epsilon < 1$ we conclude that if (QLD)$_\epsilon$ is satisfied then
	$$
	\lambda - \frac{2}{3d}\cdot \epsilon^2 + \frac{2}{3}\epsilon^4 \leq \vec{d}_{B,x}(x) \leq \lambda + \frac{2}{3d}\cdot \epsilon^2 + \frac{1}{3d} \cdot \epsilon^3
	$$ from where the result immediately follows.
\end{proof}

\subsection{(QLD) implies $(P)_K$}

Having the estimates from the previous section, we are now ready to prove the following result.

\begin{proposition}
	\label{prop:qldpk} If $\P$ verifies (QLD)$_\epsilon$ for some $\epsilon \in (0,1)$ then $(P)_K$ is satisfied for any $K\geq 15d+5$.
\end{proposition}

Proposition \ref{prop:qldpk} follows from the validity under (QLD) of the so-called \textit{Kalikow's condition}. Indeed, if we define the coefficient
$$
\varepsilon_{\mathcal{K}}:= \inf\{ \vec{d}_{B,0}(y) \cdot e_1 : B \subsetneq \Z^d \text{ connected with }0 \in B\,,\,y \in B\}.
$$ then Kalikow's condition is said to hold whenever $\varepsilon_{\mathcal{K}} > 0$. It follows from \cite[Theorem 2.3]{SZ99}, \cite[Proposition 1.4]{Sz00} and \cite[Corollary 1.5]{Sz02} that Kalikow's condition implies condition $(T)$. On the other hand, from the discussion in Section \ref{sec:bal} we know that $(T)$ implies $(P)_K$ for $K \geq 15d+5$ so that, in order to prove Proposition \ref{prop:qldpk}, it will suffice to check the validity of Kalikow's condition. But it follows from  Proposition \ref{prop:kalibound} that, under (QLD)$_\epsilon$ for some $\epsilon \in (0,1)$, for each connected $B \subsetneq \Z^d$ and $y \in B$ we have
$$
\vec{d}_{B,0}(y) \geq \lambda - \frac{\epsilon^2}{d} \geq \frac{d-1}{d}\epsilon^2 > 0
$$ so that Kalikow's condition is immediately satisfied and thus Proposition \ref{prop:qldpk} is proved.

Alternatively, one could show Proposition \ref{prop:qldpk} by checking the polynomial condition directly by means of Kalikow's walk. Indeed, if $\kappa$ denotes the uniform ellipticity constant of $\P$ then $\omega^x_{B}(y,e) \geq \kappa$ for all connected subsets $B \subsetneq \Z^d$, $x,y \in B$ and $e \in V$. In particular, it follows from this that $P_{x,\omega^x_{B_M}}(T_{B_M}<+\infty)=1$ for all $x \in B_M$ and boxes $B_M$ as in Section \ref{section2}. Corollary \ref{cor:kali} then \mbox{shows that,} in order to obtain Proposition \ref{prop:qldpk}, it will suffice to prove the following lemma.

\begin{lemma} If $\P$ verifies (QLD)$_\epsilon$ for some $\epsilon \in (0,\frac{1}{\sqrt{2(d-1)}})$ then for each $K \in \N$ we have  
	$$ 
\sup_{x \in B^*_M} P_{x,\omega^x_{B_M}}\left( X_{T_{B_M}} \notin \partial_+ B_M\right) \leq \frac{1}{M^K}
$$ if $M \in \N$ is taken sufficiently large (depending on $K$).
\end{lemma}

\begin{proof} Notice that for each $x \in B^*_M$ we have	\begin{equation}\label{eq:qldpk}
	P_{x,\omega^x_{B_M}}\left( X_{T_{B_M}} \notin \partial_+ B_M\right) \leq P_{x,\omega^x_{B_M}}\left( X_{T_{B_M}}  \in \partial_l B_M\right) + P_{x,\omega^x_{B_M}}\left( X_{T_{B_M}} \in \partial_- B_M\right)
	\end{equation} so that it will suffice to bound each term on the right-hand side of \eqref{eq:qldpk} uniformly in $B^*_M$.

To bound the first term, we define the quantities
$$
n_+:=\#\{ n \in \{1,\dots,T_{B_M}\} : X_n - X_{n-1} = e_1 \} \text{ and }n_-:=\#\{ n \in \{1,\dots,T_{B_M}\} : X_n - X_{n-1} = -e_1\}.
$$ and notice that on the event $\{X_{T_{B_M}}  \in \partial_l B_M\}$ we must have $n_+ - n_l \leq \frac{M}{2}$ since otherwise $X$ would reach $\partial_+ B_M$ before $\partial_l B_M$. Furthermore, on this event we also have \mbox{that $T_{B_M} \geq 24M^3$ since,} \mbox{by definition of $B^*_M$,} starting from any $x \in B^*_M$ it takes $X$ at least $24M^3+1$ steps to reach $\partial_l B_M$. It then follows that
\begin{equation}
\label{eq:qdl1}
P_{x,\omega^x_{B_M}}\left( X_{T_{B_M}}  \in \partial_l B_M\right)  =  \sum_{n=24M^3}^{\infty} P_{x,\omega^x_{B_M}}\left( n_+-n_- \leq \frac{M}{2} ,T_{B_M}=n\right).
\end{equation} Now, observe that the right-hand side of \eqref{eq:qdl1} can be bounded from above by
$$
\sum_{n=24M^3} \left[P_{x,\omega^x_{B_M}}\left( n_+ + n_- \leq \kappa N ,T_{B_M}=n\right) + 
P_{x,\omega^x_{B_M}}\left( n_+ + n_- > \kappa n , n_+ - n_- \leq \frac{M}{2} , T_{B_M}=n\right)\right].
$$ But since for all $y \in B_M$ and $e \in V$ we have $\omega^x_{B_M}(y,e) \geq \kappa=\frac{1}{4d}$ by the uniform ellipticity of $\P$ and, furthermore, by Proposition \ref{prop:kalibound} 
\begin{equation}
\label{eq:kali4}
\frac{\omega^x_{B_M}(y,e_1)-\omega^x_{B_M}(y,-e_1)}{\omega^x_{B_M}(y,e_1)+\omega^x_{B_M}(y,-e_1)}=\frac{\vec{d}_{{B_M},x}(y)\cdot e_1}{\omega^x_{B_M}(y,e_1)+\omega^x_{B_M}(y,-e_1)} \geq \frac{\lambda - \frac{\epsilon^2}{d}}{2\kappa} \geq 2(d-1)\epsilon^2 > 0,
\end{equation} it follows by coupling with a suitable random walk (with i.i.d. steps) that for \mbox{$n \geq 24M^3$ and $\epsilon < \frac{1}{\sqrt{2(d-1)}}$} (so as to guarantee that $\frac{1}{2}+(d-1)\epsilon^2 \in (0,1))$ we have
$$
P_{x,\omega^x_{B_M}}\left( n_+ + n_- \leq \kappa n ,T_{B_M}=n\right) \leq F(\kappa n ; n,2\kappa)
$$ and 
$$
P_{x,\omega^x_{B_M}}\left( n_+ + n_- > \kappa n , n_+ - n_- \leq \frac{M}{2} , T_{B_M}=n\right) \leq F\left(\frac{\kappa}{2}n +\sqrt[3]{n};\kappa n , \frac{1}{2}+(d-1)\epsilon^2\right) 
$$ where $F(t;k,p)$ denotes the cumulative distribution function of a $(k,p)$-Binomial random variable evaluated at $t \in \R$. By using Chernoff's bound which states that for $t \leq np$
$$
F(t;n,p) \leq \exp\left\{ -\frac{1}{2p}\cdot\frac{(np-t)^2}{n}\right\}
$$ we may now obtain the desired polynomial decay for this term, provided that $M$ is large enough (as a matter of fact, we get an exponential decay in $M$, with a rate which depends on $\kappa$ and $\epsilon$).

To deal with second term in the right-hand side of \eqref{eq:qldpk}, we define the sequence of stopping times $(\tau_n)_{n \in \N_0}$ by setting
$$
\left\{\begin{array}{ll} \tau_0:= 0 \\ \tau_{n+1}:= \inf \{ n > \tau_n : (X_n - X_{n-1}) \cdot e_1 \neq 0\} \wedge T_{B_M}\end{array}\right.
$$ and consider the auxiliary chain $Y=(Y_k)_{k \in \N_0}$ given by
$$
Y_k:= X_{\tau_k} \cdot e_1
$$ It follows from its definition and \eqref{eq:kali4} that $Y$ is a one-dimensional random walk with a probability of jumping right from any position which is at least $\frac{1}{2} +(d-1)\epsilon^2$. Now recall that, for any random walk on $\Z$ starting from $0$ with nearest-neighbor jumps which has a probability $p \neq \frac{1}{2}$ of jumping right from any position, given $a,b \in \N$ the probability $E(-a,b,p)$ of this walk exiting the interval $[-a,b]$ through $-a$ is exactly 
$$
E(-a,b,p)=\frac{1-\left(\frac{1-p}{p}\right)^b}{1-\left(\frac{1-p}{p}\right)^{a+b}}\cdot \left(\frac{1-p}{p}\right)^a.
$$ Thus, we obtain that 
$$
P_{x,\omega^x_{B_M}}\left( X_{T_{B_M}}  \in \partial_- B_M\right) \leq E\left(-M,\frac{M}{2},\frac{1}{2}+(d-1)\epsilon^2\right) =\frac{1-\left(\frac{1-2(d-1)\epsilon^2}{1+2(d-1)\epsilon^2}\right)^{\frac{M}{2}}}
{1-\left(\frac{1-2(d-1)\epsilon^2}{1+2(d-1)\epsilon^2}\right)^{\frac{3}{2}M}}\cdot \left(\frac{1-2(d-1)\epsilon^2}{1+2(d-1)\epsilon^2}\right)^M
$$ from where the desired polynomial decay (in fact, exponential) for this second term now follows.$  $ This concludes the proof.
\end{proof}

\subsection{Finishing the proof of Theorem \ref{theo:0}} We now show how to conclude the proof of \mbox{Theorem \ref{theo:0}}.

First, we observe that by Proposition \ref{prop:qldpk} the polynomial condition $(P)_K$ holds for all $K \geq 15d+5$ if (QLD)$_\epsilon$ is satisfied for $\epsilon$ sufficiently small, so that by Proposition \ref{prop-time} we have in this case that
$$
\lim_{n \rightarrow +\infty}\frac{E_0(T_n)}{n}=\frac{1}{\vec{v}\cdot e_1}.
$$ On the other hand, it follows from Proposition \ref{prop:kalibound} that if for each $n \in \N$ we define the hyperplane 
$$
B_n:= \{x \in \Z^d : x \cdot e_1 \leq n\},
$$ then $P_{0,B_n^0}(T_{B_n}<+\infty)=1$. Indeed, Proposition \ref{prop:kali} yields that $\inf_{y \in B_n} \vec{d}_{B_n,0}(y) \cdot e_1 > \frac{1}{d}\epsilon^2 > 0$ which, for example by suitably coupling Kalikow's walk with a one-dimensional walk with i.i.d. steps and a drift $\frac{1}{d}\epsilon^2$ to the right, yields our claim. Hence, by Corollary \ref{cor:kali} we obtain that
\begin{equation}
\label{eq:lastlim}
\lim_{n \rightarrow +\infty}\frac{E_{0,\omega^0_{B_n}}(T_n)}{n}=\frac{1}{\vec{v}\cdot e_1}.
\end{equation} Now, noting that for each $n\in \N$ the stopped process $M^{(n)}=(M^{(n)}_k)_{k \in \N_0}$ defined as
$$
M^{(n)}_k:= X_{k \wedge T_{n}} - \sum_{j=1}^{k \wedge T_n} \vec{d}_{S_n,0}(X_{j-1})
$$ is a mean-zero martingale under $P_{0,\omega^0_{S_n}}$, by Proposition \ref{prop:kalibound} and the \mbox{optional stopping theorem for $M^{(n)}$,} we conclude that
$$
n=E_{0,\omega^0_{S_n}}(X_{T_{n}}\cdot e_1)  \leq \left(\lambda + \frac{\epsilon^2}{d}\right)E_{0,\omega^0_{S_n}}(T_n)
$$ and analogously that
$$
n=E_{0,\omega^0_{S_n}}(X_{T_{n}}\cdot e_1) \geq \left( \lambda - \frac{\epsilon^2}{d}\right) E_{0,\omega^0_{S_n}}(T_{n}).
$$ Together with \eqref{eq:lastlim}, these inequalities imply that
$$
|\vec{v}\cdot e_1 - \lambda| \leq \frac{\epsilon^2}{d}
$$ from where the result now follows.

\section{Proof of Theorem \ref{theorem1} (Part I): seed estimates}
\label{sers}

We now turn to the proof of Theorem \ref{theorem1}. Let us observe that, having already proven Theorem \ref{theo:0} which is a stronger statement for dimension $d=2$, it suffices to show Theorem \ref{theorem1} only for $d \geq 3$. The main element in the proof of this result will be a renormalization argument, to be carried out in Section \ref{parti}. In this section, we establish two important estimates which will serve as the input for this renormalization scheme. More precisely, this section is devoted to proving the following result. As noted earlier, we assume throughout that $d \geq 3$. 

\begin{theorem}
	\label{polynomial-satisfied} If $d \geq 3$ then for any $\eta \in (0,1)$ and $\delta \in (0,\eta)$ there exist $c_{2},c_3,c_4> 0$ and $\theta_0 \in (0,1)$ depending only on $d,\eta$ and $\delta$ such that if:
	\begin{enumerate}
		\item [i.] The constant $\theta$ from \eqref{defL} is chosen smaller than $\theta_0$,
		\item [ii.] (LD)$_{\eta,\epsilon}$ is satisfied for $\epsilon$ sufficiently small depending only on $d,\eta,\delta$ and $\theta$,
	\end{enumerate}
	then 
	\begin{equation}\label{eq:poly}
	\sup_{x \in  B^*_{NL}}P_x\left(X_{T_{B_{NL}}} \notin\partial_+B_{NL}\right)\le e^{-c_{2}\epsilon^{-1}}
	\end{equation} and
	\begin{equation}\label{eq:teo6}
	\mathbb P\left(\left\{ \omega \in \Omega :\sup_{x \in \partial_- B^*_{NL}}\left|\frac{E_{x,\omega}\left(T_{B_{NL}}\right)}{NL/2}- \frac{1}{\lambda}
	\right|>
	\frac{c_{4}}{\lambda^2}\epsilon^{\alpha(d)-\delta}
	\right\}\right)\le e^{-c_{3}\epsilon^{-\delta}}.
	\end{equation}
	\end{theorem}

We divide the proof of this result into a number of steps, each occupying a separate subsection.

\subsection{(LD) implies $(P)_K$}\label{sec:poly} The first step in the proof will be to show \eqref{eq:poly}. Notice that, in particular, \eqref{eq:poly} tells us that for any $K \ge 1$ the polynomial condition $(P)_K$ is satisfied if $\epsilon$ is sufficiently small. This fact will also be important later on. The general strategy to prove \eqref{eq:poly} is basically to exploit the estimates obtained in \cite{Sz03} to establish the validity of the so-called effective criterion. First, let us consider the box $B$ given by 
\begin{equation}\label{eq:defb}
B:=(-NL,NL) \times \left(-\frac{1}{4}(NL)^3,\frac{1}{4}(NL)^3\right)^{d-1}
\end{equation} and define all its different boundaries $\partial_i B$ for $i=+,-,l$ by analogy with Section \ref{sec:GN}. Observe that \mbox{if for $x \in B^*_{NL}$} we consider $B(x):=B+x$, i.e the translate of $B$ centered at $x$, then by choice of $B$ we have that for any $\omega \in \Omega$
\begin{equation}\label{eq:pcb}
P_{x,\omega}\left( X_{T_{B_{NL}}} \notin \partial_+ B_{NL}\right) \leq P_{x,\omega}\left( X_{T_{B(x)}} \notin \partial_+ B(x)\right).
\end{equation} Thus, from the translation \mbox{invariance of $\P$} it follows that to obtain \eqref{eq:poly} it will suffice to show that 
\begin{equation}\label{eq:pc1}
P_0\left( X_{T_B} \notin \partial_+ B\right) \leq e^{-c_{2}\epsilon^{-1}}
\end{equation}for some constant $c_{2}=c_{2}(d,\eta)> 0$ if $\epsilon$ is sufficiently small. To do this, we will exploit the results developed in \cite[Section 4]{Sz03}. Indeed, if for $\omega \in \Omega$ we define
$$
q_B(\omega):=P_{0,\omega}\left( X_{T_B} \notin \partial_+ B\right) \hspace{1cm}\text{ and }\hspace{1cm}\rho_B(\omega):=\frac{q_B(\omega)}{1-q_B(\omega)}
$$ then observe that 
$$
P_0(X_{T_B} \notin \partial_+ B) = \E(q_B) \leq \E(\sqrt{q_B}) \leq \E(\sqrt{\rho_B}).
$$ But the results from \cite[Section 4]{Sz03} show that there exists a constant $c > 0$ and $\theta_0(d), \tilde{\epsilon}_0(d,\eta)>0$ such that if (LD)$_{\eta,\epsilon}$ is satisfied for $\epsilon \in (0,\tilde{\epsilon}_0)$ and $L$ from \eqref{defL} is given by $L=2[\theta \epsilon^{-1}]$ with $\theta \in (0,\theta_0)$ then 
$$
\E(\sqrt{\rho_B}) \leq \frac{80}{c\epsilon^{\alpha(d)-\eta}L}\exp\left(-\frac{c}{20}\epsilon^{\alpha(d)-\eta}NL\right) + 2d\exp\left(NL \left[\frac{\log 4d}{2}- 50\frac{\log 4d}{\log 2}\left(\frac{3}{4}-\frac{7}{100}\right)^2\right]\right).
$$ However, since for $\epsilon < \frac{\theta}{2}$ we have that
$$
\epsilon^{\alpha(d)-\eta}NL \geq 16\epsilon^{\alpha(d)-\eta}(\theta\epsilon^{-1}-1)^4 \geq \theta^4 \epsilon^{\alpha(d)-\eta-4} \geq \theta^4 \epsilon^{-(1+\eta)}
$$ together with
$$
\epsilon^{\alpha(d)-\eta}L \geq \theta \epsilon^{\alpha(d)-\eta-1} \geq \theta \epsilon
$$ and
$$
\frac{\log 4d}{2}- 50\frac{\log 4d}{\log 2}\left(\frac{3}{4}-\frac{7}{100}\right)^2 < 0,
$$ it is straightforward to check that if $\epsilon < \epsilon_0(d,\eta,\theta)$ then \eqref{eq:pc1} is satisfied.

\subsection{Exit measure from small slabs}
The second step is to obtain a control on the probability
that the random walk exits the slab $U$ ``to the right''.
For this we will follow to some extent Section 3 of Sznitman
\cite{Sz03}. 
We begin by giving two estimates: first, a bound for
the (annealed) expectation of $G_U(\vec{d}(0)\cdot e_1)$ in terms of the annealed expectation of $T_U$, 
and then a bound in $\P$-probability for the fluctuations of $E_{0,\omega}(T_U)$ around its mean $E_0(T_U)$.

\begin{proposition}
\label{green-estimates} If $d\ge 3$ and $\epsilon \in (0,\frac{1}{8d})$ then there exist positive constants $c_5,c_6,c_7$ and $c_8$ \mbox{such that} if $\epsilon,\theta \in (0,1)$ are such that $L \geq 2$ and $\epsilon L \leq c_5$ then one has 
\begin{equation}
\label{green-expectation}
\left|\mathbb E\left(G_{U}[\vec{d}\cdot e_1](0)\right)-
\lambda E_{0}(T_{U})\right|\le
c_6\epsilon \log L,
\end{equation}
and also 
\begin{equation}
\label{exit-time-estimate}
c_7L^2\le E_0\left(T_{U}\right)\le c_8L^2.
\end{equation} Furthermore, given any $\eta \in (0,1)$ there exists $\epsilon_0=\epsilon_0(d,\eta,\theta) \in (0,1)$ such that if (LD)$_{\eta,\epsilon}$ is satisfied for $\epsilon \in (0,\epsilon_0)$ then 
\begin{equation}
\label{eq:cotaGU}
\mathbb E\left(G_{U}[\vec{d}\cdot e_1](0)\right) \geq \frac{2}{5}d \lambda L^2.
\end{equation}
\end{proposition}

\begin{proof} A careful inspection of the proof of
\cite[Proposition 3.1]{Sz03} yields the estimates \eqref{green-expectation} and \eqref{eq:cotaGU}.
On the other hand, inequalities (2.28) and (3.6) of \cite{Sz03}
give us the bounds in \eqref{exit-time-estimate}.
\end{proof}

The next estimate we shall need is essentially contained in Proposition 3.2 of \cite{Sz03}, which gives a control on the difference between the random variable $G_U(\vec{d}(0)\cdot e_1)$ and its expectation for $d \geq 3$. We include it here for completeness and refer to \cite{Sz03} for a proof.

\begin{proposition} 
\label{control} If $d \geq 3$ then there exist constants $c_9,c_{10}>0$ such that if $\epsilon,\theta,\alpha \in (0,1)$ satisfy $L \geq 2$ and $\epsilon L < \frac{1}{2}\cdot \frac{1-\alpha}{2-\alpha} \cdot c_9$, one has for all $u \geq 0$ that
\begin{equation}
\label{green-variance}
\P\left[\left|
G_U[\vec{d}(\cdot,\omega)\cdot e_1](0)-\E(G_U[\vec{d}\cdot e_1](0))\right|\ge u\right]
     \le  c_{10}\exp\left\{-\frac{u^2}{ c_{\alpha,L}}\right\},
\end{equation}
where
$$
c_{\alpha,L}:=c_{11}\epsilon^2\sum_{y\in U}g_{0,U}(0,y)^{2/(2-\alpha)}
$$ for some constant $c_{11}=c_{11}(d) > 0$ and 
\begin{equation}
\nonumber
c_{\alpha,L}\le
\begin{cases}
 c_{1,2}\epsilon^{2}L^{1+(2(1-\alpha)/(2-\alpha))} & \quad{\rm for}\ d=3\\
 c_{1,2}\epsilon^{2}L^{(4(1-\alpha)/(2-\alpha))} & \quad{\rm for}\ d=4\\
c_{1,2} \epsilon^{2} & \quad{\rm for}\ d\ge 5\ {\rm and}\ \alpha\ge\frac{4}{5}
\end{cases}
\end{equation}
for some $c_{1,2}=c_{1,2}(\alpha,d)> 0$.
\end{proposition}

Finally, we establish a control of the fluctuations
of the quenched expectation $E_{0,\omega}(T_U)$ analogous to
the one obtained in Proposition \ref{control}.

\begin{proposition} 
\label{control2} 
If $d \geq 3$ then for any $\alpha \in [0,1)$ and $\epsilon,\theta \in (0,1)$ with $L \geq 2$ and $\epsilon L < \frac{1}{2}\cdot \frac{1-\alpha}{2-\alpha} \cdot c_9$ where $c_9$ is the constant from Proposition \ref{control}, one has for all $u \geq 0$ that

\begin{equation}
\label{time-variance}
\P\left[\left|
E_{0,\omega}(T_U)-E_0(T_U)\right|\ge u\right]
     \le  c_{12}\exp\left\{-\frac{u^2}{ c'_{\alpha,L}}\right\}
\end{equation}
for some $c_{12}=c_{12}(d) > 0$, where 
$$
c'_{\alpha,L}:=c_{13}\sum_{y\in U}g_{0,U}(0,y)^{2/(2-\alpha)}
$$ for some constant $c_{13}=c_{13}(d)>0$ and
\begin{equation}
\nonumber
c'_{\alpha,L}\le
\begin{cases}
 c'_{1,2}L^{1+(2(1-\alpha)/(2-\alpha))} & \quad{\rm for}\ d=3\\
 c'_{1,2}L^{(4(1-\alpha)/(2-\alpha))} & \quad{\rm for}\ d=4\\
c'_{1,2}  & \quad{\rm for}\ d\ge 5\ {\rm and}\ \alpha\ge\frac{4}{5}
\end{cases}
\end{equation}
for some $c'_{1,2}=c'_{1,2}(\alpha,d)>0$.
\end{proposition}

\begin{proof} We follow the proof of \cite[Proposition 3.2]{Sz03}, using the martingale method introduced there.
Let us first enumerate the elements of $U$ as $\{x_n:n \in \N\}$.
Now define the filtration
$$
\mathcal G_n:=
\begin{cases}
\sigma(\omega(x_1),\ldots,\omega(x_n))& \qquad \text{ if }n\ge 1\\
\{\emptyset,\Omega\}&\qquad \text{ if }n=0
\end{cases}
$$
and also the bounded $\mathcal G_n$-martingale $(F_n)_{n \in \N_0}$ given for each $n \in \N_0$ by 
$$
F_n:=\mathbb E\left(G_U[\mathbf{1}](0)|\mathcal G_n\right)
$$ where $\mathbf{1}$ is the function constantly equal to $1$, i.e. $\mathbf{1}(x)=1$ for all $x \in \Z^d$. Observe that 
$$
G_U[\mathbf{1}](0,\omega)=\E_{0,\omega}(T_U)
$$ by definition of $G_U$. Thus, if we prove that for all $n \in \N$
\begin{equation}
\label{efen}
|F_n-F_{n-1}|\le c_{14}g_{0,U}(0,x_n)^{\frac{1}{2-\alpha}}=:\gamma_n
\end{equation}
for some $c_{14}=c_{14}(d)>0$ then, since $F_0=E_0(T_U)$ and $F_\infty=E_{0,\omega}(T_U)$, by using Azuma's inequality and the bound for $c_{\alpha,L}$ in Proposition \ref{control} (see the proof of \cite[Proposition 3.2]{Sz03} for further details) we obtain \eqref{time-variance} at once. In order to prove \eqref{efen}, for each
$n \in \N$ and all environments $\omega,\omega'\in\Omega_\epsilon$ coinciding at every $x_i$ with $i \neq n$ define
$$
\Gamma_n(\omega,\omega'):=G_U[\mathbf{1}](0,\omega')-G_U[\mathbf{1}](0,\omega).
$$
Since $F_n-F_{n-1}$ can be expressed as an integral of $\Gamma_n(\omega,\omega')$
with respect to $\omega$ and $\omega'$, it is enough to prove
that $\Gamma_n(\omega,\omega')$ is bounded from above by the constant $\gamma_n$ from \eqref{efen}. To do this, we introduce for $u\in [0,1]$ the environment
$\omega_u$ defined for each $i \in \N$ by
$$
\omega_u(x_i)=(1-u)\cdot \omega(x_i)+u \cdot \omega'(x_i).
$$ If we set 
$$
H_{x_n}:= \inf \{ j \geq 0 : X_j = x_n \} \hspace{1cm}\text{ and }\hspace{1cm}\overline{H}_{x_n}:=\inf \{ j \geq 1 : X_j = x_n\}
$$ then, by the strong Markov property for the stopping time $H_{x_n}$, a straightforward computation yields that
\begin{equation}\label{eq:GU2}
G_U[\mathbf{1}](0,\omega_u)= E_{0,\omega_u}\left( H_{x_n}\wedge (T_U-1) + 1 \right) + P_{0,\omega_u}(H_{x_n} < T_U ) E_{x_n,\omega_u}(T_U).
\end{equation} Similarly, by the strong Markov property for the stopping time $\overline{H}_{x_n}$ we have
\begin{equation}\label{eq:GU3}
E_{x_n,\omega_u}(T_U)= E_{x_n,\omega_u}(\overline{H}_{x_n} \wedge (T_U-1) + 1) + P_{x_n,\omega_u}(\overline{H}_{x_n}< T_U)E_{x_n,\omega_u}(T_U),
\end{equation} so that 
\begin{equation}\label{eq:GU}
G_U[\mathbf{1}](0,\omega_u) = E_{0,\omega_u}\left( H_{x_n}\wedge (T_U-1)+1\right) + \frac{P_{0,\omega_u}(H_{x_n}< T_U)}{P_{x_n,\omega_u}(\overline{H}_{x_n} > T_U)}E_{x_n,\omega_u}(\overline{H}_{x_n} \wedge (T_U-1)+1).
\end{equation}
Notice that $P_{0,\omega_u}(H_{x_n}<T_U)$ and the first term in the right-hand side of \eqref{eq:GU} do not depend on $u$. Furthermore, by the Markov property for time $j=1$, we have
$$
P_{x_n,\omega_u}(\overline{H}_{x_n}>T_U) = \sum_{e \in V} \omega_u(x_n,e)P_{x_n+e,\omega_u}(H_{x_n} > T_U)
$$
and 
$$
E_{x_n,\omega_u}(\overline{H}_{x_n} \wedge (T_U-1)+1)= \sum_{e \in V} \omega_u(x_n,e) (1 + E_{x_n+e,\omega_u}(H_{x_n} \wedge (T_U-1)+1)),
$$ so that differentiating $G_U[\mathbf{1}](0,\omega_u)$ with respect to $u$ yields
$$
\partial_u G_U[\mathbf{1}](0,\omega_u) = \frac{P_{0,\omega_u}(H_{x_n}< T_U)}{P_{x_n,\omega_u}(\overline{H}_{x_n} > T_U)} \sum_{e \in V} (\omega'(x_n,e)-\omega(x_n,e)) (A_e - B_e)
$$
where
$$
A_e:=1 + E_{x_n+e,\omega_u}(H_{x_n} \wedge (T_U-1)+1)
$$ and 
$$
B_e:= \frac{P_{x_n+e,\omega_u}(H_{x_n}> T_U)}{P_{x_n,\omega_u}(\overline{H}_{x_n} > T_U)}E_{x_n,\omega_u}(\overline{H}_{x_n} \wedge (T_U-1)+1).
$$ Now, by a similar argument to the one used to obtain \eqref{eq:GU2} and \eqref{eq:GU3}, we have that
$$
E_{x_n+e,\omega_u}(H_{x_n}\wedge (T_U-1)+1) = G_U[\mathbf{1}](x_n+e,\omega_u) - P_{x_n+e,\omega_u}(\overline{H}_{x_n} < T_U)G_U[\mathbf{1}](x_n,\omega_u)
$$ and 
$$
P_{x_n,\omega_u}(\overline{H}_{x_n} > T_U)G_U[\mathbf{1}](x_n,\omega_u)= E_{x_n,\omega_u}(\overline{H}_{x_n} \wedge (T_U-1) + 1),
$$ from which we conclude that 
$$
A_e-B_e = 1 + \nabla_e G_U[\mathbf{1}](x_n,\omega_u),
$$ where for any bounded $f: \Z^d \to \R$, $x \in \Z^d$ and $\omega \in \Omega$ we write
$$
\nabla_e G_U[f](x,\omega_u):= G_U[f](x+e,\omega_u) - G_U[f](x,\omega_u).
$$ Furthermore, by the proof of \cite[Proposition 3.2]{Sz03} we have
$$
\frac{P_{0,\omega_u}(H_{x_n}< T_U)}{P_{x_n,\omega_u}(\overline{H}_{x_n} > T_U)} = g_U(0,x_n,\omega_u) \leq c'g_{0,U}(0,x_n)^{\frac{1}{2-\alpha}} \hspace{1cm}\text{ and }\hspace{1cm}|\nabla_e G_U[\mathbf{1}](x_n,\omega_u)| \leq c''L
$$  where $L$ is the quantifier from \eqref{defL}. Since for $\omega',\omega \in \Omega_\epsilon$ we have
$$
\sum_{e \in V}|\omega'(x_n,e)-\omega(x_n,e)| \leq \epsilon,
$$ we conclude that for all $u \in [0,1]$
$$
|\partial_u G_U[\mathbf{1}](0,\omega_u)| \leq c'g_{0,U}(0,x_n)^{\frac{1}{2-\alpha}} \epsilon (1 + c''L) \leq c_{14} g_{0,U}(0,x_n)^{\frac{1}{2-\alpha}}
$$ since $\epsilon \leq 1$ and $\epsilon L < \frac{1}{2}\cdot \frac{1-\alpha}{2-\alpha} \cdot c_9 \leq \frac{c_9}{4}$ by hypothesis. From this estimate \eqref{efen} immediately follows, which concludes the proof.
\end{proof}

\subsection{Exit measures from small slabs within a seed box} 
The next step in the proof is to show that, on average, the random walk starting from any $z \in B_{NL}$ sufficiently far away from $\partial_l B_{NL}$ moves at least $\pm L$ steps in direction $e_1$ before reaching $\partial_l B_{NL}$. The precise estimate we will need is contained in the following proposition.

\begin{proposition}
\label{time-truncation} There exist three positive constants $c_{15},c_{16}=c_{16}(d)$ and $\epsilon_0=\epsilon_0(d)$ verifying that if $\epsilon,\theta \in (0,1)$ are such that $L \geq 2$, $\epsilon L \leq c_{16}$ and $\epsilon \in (0,\epsilon_0)$, then one has that
\begin{equation}\label{eq:time-truncation}
\sup_{\omega \in \Omega_\epsilon} \left|E_{z,\omega}(T_{U_L(z)})-E_{z,\omega}(T_{U_L(z)} \wedge T_{\partial_l B_{NL}})\right|\le 
e^{-c_{15}L} 
\end{equation}
for all $z \in B_{NL}'$, where 
$$
B'_{NL}:=\left\{ z \in B_{NL} : \sup_{2 \leq i \leq d} |z \cdot e_i | \leq 25(NL)^3 - N\right\}.
$$
\end{proposition}

To prove Proposition \ref{time-truncation} we will
require the following two lemmas related to the
exit time $T_{U}$. The first lemma gives a uniform bound on the second moment of $T_{U}$.

\begin{lemma}
\label{lemmaa} There exist constants $c_{17},c_{18}=c_{18}(d)>0$ such that if $\epsilon,\theta \in (0,1)$ are taken such that $L \geq 2$ and $\epsilon  L \le c_{18}$ 
then one has that
$$
\sup_{z \in \Z^d,\,\omega \in \Omega_\epsilon}E_{z,\omega}(T_{U_L(z)}^2)\le c_{17} L^4.
$$
\end{lemma}
\begin{proof} Let us fix $x \in \Z^d$ and write $U_z:=U_L(z)$ in the sequel for simplicity. Notice that 
\begin{align}
E_{z,\omega}(T_{U_z}^2)&=E_{z,\omega}\left(\left(
\sum_{x\in U_z}\sum_{j=0}^{\infty}\mathbbm 1_x(X_{j})\mathbbm 1(j< T_{U_z})\right)^2\right)\nonumber \\
\label{expect-square}
& =
2\sum_{x\in U_z}\sum_{y\in U_z}\sum_{j=0}^{\infty}\sum_{k>j}^{\infty}
E_{z,\omega}\left(
\mathbbm 1_x(X_{j})\mathbbm 1_y( X_{k})
\mathbbm 1(k< T_{U_z})\mathbbm 1(j< T_{U_z})
\right)
+
E_{z,\omega}\left(T_{U_z}\right).
\end{align}
Now, by the Markov property, for each $j<k$ we have that
$$
E_{z,\omega}\left(\mathbbm 1_x(X_{j})\mathbbm 1_y( X_{k})
\mathbbm 1(k< T_{U_z})\mathbbm 1(j< T_{U_z})
\right)
=
E_{z,\omega}\left(
\mathbbm 1_x(X_{j})\mathbbm 1(j< T_{U_z})\right)
 E_{x,\omega}\left(\mathbbm 1_y( X_{i})
\mathbbm 1(i< T_{U_z})\right),
$$
where $i:=k-j$. Substituting this back into \eqref{expect-square}, we see that
\begin{align*}
E_{z,\omega}(T_{U_z}^2)&\le
2\sum_{x\in {U_z}}\sum_{y\in {U_z}}\sum_{j=0}^{\infty}\sum_{i=1}^{\infty}
E_{z,\omega}\left(
\mathbbm 1_x(X_{j})\mathbbm 1(j< T_{U_z})\right)
 E_{x,\omega}\left(\mathbbm 1_y( X_{i})
\mathbbm 1(i< T_{U_z})\right)+
E_{z,\omega}(T_{U_z})\\
& \le
2\sum_{x\in {U_z}}\sum_{j=0}^{\infty}
E_{z,\omega}\left(
\mathbbm 1_x(X_{j})\mathbbm 1(j< T_{U_z})\right)
 E_{x,\omega}\left(T_{U_z}\right)+
E_{z,\omega}(T_{U_z})\\
& = 2 E_{z,\omega}(T_{U_z}) \left(\sup_{y \in \Z^d} E_{y,\omega}(T_{U_z})\right) + E_{z,\omega}(T_{U_z}) \\
& \leq 2c^2L^4 + cL^2
\end{align*}
for some constant $c>0$, where for the last line we have used inequality (2.28) of Sznitman in \cite{Sz03}, which says that 
$$
\sup_{z,y \in \Z^d,\omega \in \Omega_\epsilon} E_{y,\omega}(T_{U_z}) \leq cL^2
$$ whenever $L \geq 2$ and $\epsilon L \leq c_{18}(d)$. From this the result immediately follows.
\end{proof}

Our second auxiliary lemma states that, with overwhelming probability, the random walk starting from any $x \in B_{NL}$ far enough from $\partial_l B_{NL}$ is very likely to move at least $\pm L$ steps in direction $e_1$ before reaching $\partial_l B_{NL}$.

\begin{lemma}
\label{lemmab} There exist constants $c_{19},c_{20}>0$  
such that if $\epsilon,\theta \in (0,1)$ satisfy $L \geq 2$ and $\epsilon L \leq c_{20}$ then for any $a \in (0,25(NL)^3)$ and $z\in B_{NL}$ verifying $\sup_{2 \le i\le d}|z\cdot e_i|\le a$ one has
$$
\sup_{\omega \in \Omega_\epsilon}P_{z,\omega}(T_{\partial_l B_{NL}} \leq T_{U_L(z)}) \le 2e^{-c_{19}\frac{25(NL)^3-a}{L^2}}.
$$
\end{lemma}

\begin{proof} Note that for any $z \in B_{NL}$ with $\sup_{2 \leq i \leq d}|z\cdot e_i| \leq a$ one has that 
$$
P_{z,\omega}(T_{\partial_l B_{NL}} \leq T_{U_L(z)})  \le 
P_{z,\omega}(T_{U_L(z)}\ge 25(NL)^3-a).
$$ 
Furthermore, by Proposition 2.2 in \cite{Sz03}, there exist constants $c_{19},c_{20} > 0$ such that if $\epsilon L \leq c_{20}$ then for any $z \in \Z^d$
$$
\sup_{x \in \Z^d,\omega \in \Omega_\epsilon} E_{x,\omega}\left( e^{\frac{c_{19}}{L^2}T_{U_L(z)}}\right)\leq 2.
$$ Hence, by the exponential Tchebychev inequality we conclude that
$$
P_{z,\omega}(T_{\partial_l B_{NL}} \leq T_{U_L(z)}) \le e^{-c_{19}\frac{25(NL)^3-a}{L^2}}E_{z,\omega}\left(
e^{\frac{c_{19}}{L^2}T_{U_L(z)}}\right)\le
2e^{-c_{19}\frac{25(NL)^3-a}{L^2}}. 
$$\end{proof}

We are now ready to prove Proposition \ref{time-truncation}.
Indeed, notice that
\begin{align*}
E_{z,\omega}(T_{U_L(z)})&=E_{z,\omega}(T_{U_L(z)}\mathbbm 1(T_{U_L(z)}<T_{\partial_l B_{NL}}))+
E_{z,\omega}(T_{U_L(z)}\mathbbm 1(T_{U_L(z)}\ge T_{\partial_l B_{NL}}))\\
&
=E_{z,\omega}((T_{U_L(z)} \wedge T_{\partial_l B_{NL}})\mathbbm{1}(T_{U_L{(z)}}<T_{\partial_l B_{NL}}))
+
E_{z,\omega}(T_{U_L(z)}\mathbbm 1(T_{U_L(z)}\ge T_{\partial_l B_{NL}}))\\
& \leq E_{z,\omega}(T_{U_L(z)} \wedge T_{\partial_l B_{NL}}) +E_{z,\omega}(T_{U_L(z)}\mathbbm 1(T_{U_L(z)}\ge T_{\partial_l B_{NL}})).
\end{align*}
Hence, since $\sup_{2 \leq i \leq d}|z \cdot e_i| \leq 25(NL)^3 -N$ for any $z \in B'_{NL}$, by the Cauchy-Schwarz inequality and Lemmas \ref{lemmaa} and \ref{lemmab} it follows that
$$
|E_{z,\omega}(T_{U_L(z)})-E_{z,\omega}(T_{U_L(z)} \wedge T_{\partial_l B_{NL}})|\le \sqrt{E_{z,\omega}(T_{U_L(z)}^2)
P_{z,\omega}(T_{U_L(z)}\ge T_{\partial_l B_{NL}})}
\le \sqrt{2c_{17}}L^2 e^{-\frac{c_{19}}{2}L}.
$$ From this estimate, taking $c_{16}:=\min\{c_{18},c_{20}\}$ and $\epsilon$ sufficiently small yields \eqref{eq:time-truncation}.

\subsection{Renormalization scheme to obtain a seed estimate}
 Our next step is to derive estimates on the time spent by the random walk on slabs of size $NL$.

Let us fix $\omega \in \Omega$ and define two sequences $W=(W_k)_{k \in \N_0}$ and $V=(V_k)_{k \in \N_0}$ of stopping times, by setting $W_0 = 0$ and then for each $k \in \N_0$
$$
W_{k+1}:=\inf\{ n > W_k : |(X_n - X_{W_k})\cdot e_1| \geq L\} \hspace{1cm}\text{ and }\hspace{1cm}V_k:= W_{k} \wedge T_{B_{NL}}.
$$ Now, consider the random walks $Y=(Y_k)_{k \in \N_0}$ and $Z=(Z_k)_{k \in \N_0}$ defined for $k \in \N_0$ by the formula
\begin{equation}
\label{auxiliary-walk}
Y_k:=X_{W_k}
\end{equation}
and
\begin{equation}
\label{auxiliary-walk2}
Z_k:=X_{V_k}.
\end{equation}
Notice that at each step, the random walk $Y$ jumps from $x$ towards some $y$ with $(y-x)\cdot e_1 \geq L$, i.e. it exits the slab $U_L(x)$ ``to the right'', with probability $\hat{p}(x,\omega)$, where 
$$
\hat p(x,\omega):=P_{x,\omega}\left(T_{x\cdot e_1+L}<
 T_{x\cdot e_1-L}\right).
$$ Observe also that $\hat{p}$ verifies the relation
\begin{equation}
\label{eq:hatp}
\hat{p}(x,\omega)= \frac{1}{2}+\frac{1}{2L}G_{U_L(x)}[\vec{d}\cdot e_1](x,\omega)
\end{equation} which follows from an application of the optional sampling theorem to the $P_\omega$-martingale $(M_{k,\omega})_{k \in \N}$ given by 
$$
M_{k,\omega} = X_k - \sum_{j=0}^{k-1} \vec{d}(X_j,\omega).
$$ Now, for each $p \in  [0,1]$ let us couple $Y^{(e_1)}:=(Y_k \cdot e_1)_{k \in \N_0}$ with a random walk \mbox{$y^{(p)}:=(y^{(p)}_k)_{k \in \N_0}$ on $\mathbb Z$,} which starts at $0$ and in each step jumps one unit to the right with \mbox{probability $p$} and one to the left with probability $1-p$, in such a way that both $Y^{(e_1)}$ and $y^{(p)}$ jump together in the \mbox{rightward direction} with the largest possible probability, i.e. for any $k \geq 0$, $x \in \Z^d$ and $m \in \Z$
$$
P_\omega( Y^{(e_1)}_{k+1} \geq x \cdot e_1 + L \,,\,y^{(p)}_{k+1}=m +1 | Y_k = x \,,\,y_k^{(p)}=m) = \min\{ \hat{p}(x,\omega),p\}.
$$ The explicit construction of such a coupling is straightforward, so we omit the details. Call this the \textit{coupling to the right} of $Y^{(e_1)}$ and $y^{(p)}$. Now, consider the random walks $y^-:=y^{(p_-)}$ and $y^+:=y^{(p_+)}$, where 
\begin{equation}\label{p-}
p_-:= \left(\frac{1}{2}+\frac{\mathbb E(G_U[\vec{d}\cdot e_1](0))-\epsilon^{\alpha(d)-2-\delta}}{2L} \right)\vee 0
\end{equation}and
\begin{equation}\label{p+}
p_+:=\left(\frac{1}{2}+\frac{\mathbb E(G_U[\vec{d}\cdot e_1](0))+\epsilon^{\alpha(d)-2-\delta}}{2L}\right) \wedge 1,
\end{equation} and assume that they are coupled with $Y^{(e_1)}$ to the right. Let us call $E_0^-$ and $E_0^+$ the expectations defined by their
respective laws. 
Next, for each $M \in \N$ define the stopping times $\mathcal{T}^Y_M$, $S^+_M$ and $S^-_M$ given by 
$$
\mathcal T^Y_M:=\inf\left\{k \ge 0:Y_k\cdot e_1 \geq LM\right\} \hspace{1cm}\text{ and }\hspace{1cm}
\mathcal S^\pm_M:=\inf\left\{k\ge 0:y_k^\pm \geq M\right\}.
$$
Finally, if for each subset $A\subset\mathbb Z^d$ we define
the stopping time
$$
\mathcal T^Z_A:=\inf\left\{k\ge 0:Z_k\notin A\right\},
$$ we have the following control on the expectation of $\mathcal{T}^Z_{B_{NL}}$. 

\begin{proposition}
\label{exit-time-z} If $d \geq 3$ then for any given $\eta \in (0,1)$ and $\delta \in (0,\eta)$ there exist $c_{21},c_{22} > 0$ and $\theta_0 \in (0,1)$ depending only on $d,\eta$ and $\delta$ such that if:
\begin{enumerate}
	\item [i.] The constant $\theta$ from \eqref{defL} is chosen smaller than $\theta_0$,
	\item [ii.] (LD)$_{\eta,\epsilon}$ is satisfied for $\epsilon$ sufficiently small depending only on $d,\eta,\delta$ and $\theta$,
\end{enumerate} then for any $z \in \partial_- B^*_{NL}$ we have
$$
\mathbb P\left(\left\{ \omega \in \Omega : \frac{N/2}{2p_+ -1}-e^{-c_{22}\epsilon^{-1}}
\le E_{z,\omega}\left(\mathcal{T}^Z_{B'_{NL}}\right)\leq E_{z,\omega}\left(\mathcal T^Z_{B_{NL}}\right)\le 
\frac{N/2}{2p_--1}\right\}\right)\ge
1-e^{-c_{21}\epsilon^{-\delta}}.
$$ 
\end{proposition}

\begin{proof} 
Define the event
\begin{equation}
\label{devent}
\mathcal B:=\bigcap_{x\in B_{NL}}\left\{ \omega \in \Omega : \left|
G_U[\vec{d}\cdot e_1](x,\omega)-\mathbb E(G_U[\vec{d}\cdot e_1](x))\right|\le
 \epsilon^{\alpha(d)-2-\delta}
\right\}.
\end{equation}
Let us observe that for any $\omega \in \mathcal{B}$ we have $p_- \leq \hat{p}(x,\omega)$ for all $x \in B_{NL}$. In particular, since $Y^{(e_1)}$ is coupled to the right with $y^-$, if $Y^{(e_1)}_0 = \frac{NL}{2}$ and $y_0^- =0$ then for any $\omega \in \mathcal{B}$ we have
$$
Y^{(e_1)}_k \ge L y_k^- + \frac{NL}{2}
$$ for all $0 \leq k \leq \mathcal{T}^Z_{B_{NL}}$ so that, in particular, for any $\omega\in\mathcal B$  
$$
\mathcal T^Z_{B_{NL}}\le S^-_{\frac{N}{2}}
$$
and thus
$$
E_{z,\omega}\left(\mathcal T^Z_{B_{NL}}\right)\le E^-_0(S^-_{\frac{N}{2}}).
$$
Similarly, since $Y^{(e_1)}$ is coupled to the right with $y^+$ and $\hat{p}(x,\omega)\leq p_+$ for all $x \in B_{NL}$ when $\omega \in \mathcal{B}$, if $y^+_0=0$ then for any $\omega \in \mathcal{B}$ we have 	
$$
Y_k^{(e_1)} \leq L y_k^+ + \frac{NL}{2}
$$ for all $0 \leq k \leq \mathcal{T}^Z_{B_{NL}}$, so that for any such $\omega$ on the event $\{\mathcal{T}^Y_{NL} = \mathcal{T}^Z_{B'_{NL}}\}$ we have 
$$
\mathcal{T}^Z_{B'_{NL}} \geq S^+_{\frac{N}{2}}.
$$ Therefore, we see that for each $z \in \partial_- B^*_{NL}$
\begin{align}
E_{z,\omega}\left(\mathcal T^Z_{B'_{NL}}\right) &=
E_{z,\omega}\left(\mathcal T^Z_{B'_{NL}}
\mathbbm 1( 
\mathcal T^Z_{B'_{NL}}<
\mathcal T^Y_{NL})\right)
+
E_{z,\omega}\left(\mathcal T^Z_{B'_{NL}}
\mathbbm 1( \mathcal T^Y_{NL}= \mathcal T^Z_{B'_{NL}})\right) \nonumber\\
& \geq 
E_{z,\omega}\left(S^+_{\frac{N}{2}}
\mathbbm 1( \mathcal T^Y_{NL}= \mathcal T^Z_{B'_{NL}})\right) \nonumber\\
&\label{fff} =E^+_{0}\left(S^+_{\frac{N}{2}}\right)
-
E_{z,\omega}\left(S^+_{\frac{N}{2}}
\mathbbm 1( \mathcal T^Z_{B'_{NL}}<\mathcal T^Y_{NL})\right).
\end{align}
Now, by the Cauchy-Schwarz inequality we have that
\begin{align}
E_{z,\omega}\left(S^+_{\frac{N}{2}}
\mathbbm 1( \mathcal T^Z_{B'_{NL}}<\mathcal T^Y_{NL})\right)
 &\leq \sqrt{E^+_{0}\left( \left(S_{\frac{N}{2}}^+\right)^2\right)P_{z,\omega}\left(
X_{T_{B'_{NL}}}\notin \partial_+ B'_{NL}\right)} \nonumber\\
& \leq \sqrt{E_{0}^+\left(\left(S_{\frac{N}{2}}^+\right)^2\right)P_{z,\omega}\left(
	X_{T_{B(z)}}\notin \partial_+ B(z)\right)} \label{ggg}
\end{align} where $B(z):=B+z$ for $B$ as defined in \eqref{eq:defb} and, to obtain the last inequality, we have repeated the same argument used to derive \eqref{eq:pcb} but for $B'_{NL}$ instead of $B_{NL}$ (which still goes through if $L \geq 2$).
On the other hand, using the fact that the sequences $M^\pm =(M^\pm_n)_{n \in \N_0}$ and $N^\pm=(N^\pm_n)_{n \in \N_0}$ given for each $n \in \N_0$ by 
$$
M_n^\pm=y^{\pm}_n-n(2p^\pm -1)
$$
and
$$
N_n^\pm=\left(y^{\pm}_n-n(2p^\pm -1)\right)^2-n(1-(2p^\pm-1)^2)
$$
are all martingales with respect to the natural filtration
generated by their associated random walks, and also that by Proposition \ref{green-estimates} if $\epsilon$ is sufficiently small (depending on $d,\theta,\eta$ and $\delta$)
$$
2p^\pm - 1 = \frac{1}{L} (\E( G_U[\vec{d}\cdot e_1](0)) \pm \epsilon^{\alpha(d)-\eta/2}) > 0
$$ since (LD)$_{\eta,\epsilon}$ is satisfied and $\delta < \eta$, we conclude that
$$
E^\pm_0(S_{\frac{N}{2}}^\pm)=\frac{N/2}{2p^\pm-1},
$$
and 
$$
E^+_0\left(\left(S_{N}^+\right)^2\right)=\frac{(N/2)^2}{(2p^+-1)^2} +\frac{(N/2)}{2p^+-1}(1-(2p^+-1)^2) \leq C_+ N^2
$$ if $\epsilon \in (0,1)$, where $C_+ > 0$ is a constant depending on $p^+$. Inserting these bounds in \eqref{fff} and \eqref{ggg}, we conclude that for $\omega\in\mathcal B$ one has
\begin{equation}
\label{enele}
\frac{N/2}{2p^+-1}
-
\sqrt{C_+N^2
P_{z,\omega}\left(
X_{T_{B(z)}}\notin\partial_+ B(z)\right)}\le E_{z,\omega}\left(\mathcal{T}^Z_{B'_{NL}}\right) \leq 
E_{z,\omega}\left(\mathcal T^Z_{B_{NL}}\right)\le
\frac{N/2}{2p^--1}.
\end{equation}
But, by the proof of \eqref{eq:poly} in Section \ref{sec:poly} and Markov's inequality, we have that
\begin{align}
\mathbb P\left(P_{z,\omega}\left(
X_{T_{B(z)}}\notin\partial_+ B(z)\right)\ge
e^{-\frac{1}{2}c_{2}\epsilon^{-1}}\right)
&\le
e^{\frac{1}{2}c_{2}\epsilon^{-1}}
P_{z}\left( X_{T_{B(z)}}\notin\partial_+ B(z)\right) \nonumber \\
\label{pomega}
&\le \exp\left(-\frac{1}{2}c_{2}\epsilon^{-1}\right),
\end{align} where $c_2=c_2(d,\eta) > 0$ is the constant from \eqref{eq:poly}.
Furthermore, Proposition \ref{control} implies that $\theta$ from \eqref{defL} can be chosen so that for any $\epsilon$ sufficiently small (depending on $d$, $\delta$ and $\theta$)  
\begin{equation}
\label{pbe}
\mathbb P(\mathcal B^c)\le C(d)(NL)^{3(d-1)+1}\exp\left(-c\epsilon^{-\delta}\right)
\end{equation} for some constants $C(d),c>0$.  
Combining the estimates \eqref{pomega} and \eqref{pbe}
with the inequalities in \eqref{enele}, we conclude the proof.
\end{proof}

\subsection{Proof of \eqref{eq:teo6}}

We conclude this section by giving the proof of \eqref{eq:teo6}. The proof has two steps: first, we express the expectation $E_{x,\omega}(T_{B_{NL}})$ for $x \in \partial_- B^*_{NL}$ in terms of the Green's function \mbox{of $Z$} and the quenched expectation of $T_{U_L} \wedge T_{B_{NL}}$, and then combine this with the estimates obtained in the previous subsections to conclude the result. 
The first step is contained in the next lemma.

\begin{lemma}
\label{step1} If we define $\mathcal{Z}:=\{ z \in B_{NL} : z \cdot e_1 = kL \text{ for some }k \in \Z\}$ and the Green's function
$$
g_Z(x,y,\omega):= \sum_{i=0}^\infty \E_{x,\omega}( \mathbbm{1}_{\{y\}}(Z_i)\mathbbm{1}_{\{i < \mathcal{T}^Z_{B_{NL}}\}}),
$$ where $Z$ is the random walk in \eqref{auxiliary-walk2}, then for any $x \in \partial_- B^*_{NL}$ we have that
\begin{equation}
\label{convolution2}
E_{x,\omega}\left(T_{B_{NL}}\right)=
\sum_{z\in\mathcal Z} g_Z(x,z,\omega)E_{z,\omega}\left(T_{U_L(z)}\wedge T_{B_{NL}}\right).
\end{equation} 
\end{lemma}

\begin{proof} Note that 
\begin{align*}
E_{x,\omega}\left(T_{B_{NL}}\right)&=
\sum_{y\in B_{NL}}E_{x,\omega}\left(\sum_{n=0}^\infty \mathbbm{1}_{\{y\}}(X_n)\mathbbm{1}_{\{n < T_{B_{NL}}\}}\right)\\
& =
\sum_{y\in B_{NL}}\sum_{i=0}^\infty E_{x,\omega}\left(\sum_{n=W_i}^{W_i-1}\mathbbm{1}_{\{y\}}(X_n)\mathbbm{1}_{\{n < T_{B_{NL}}\}}\right)\\
& = \sum_{y\in B_{NL}}\sum_{i=0}^\infty E_{x,\omega}\left( \mathbbm{1}_{\{W_i < T_{B_{NL}}\}}
E_{X_{W_i},\omega}\left(\sum_{n=0}^{W_1-1}\mathbbm{1}_{\{y\}}(X_n)\mathbbm{1}_{\{n < T_{B_{NL}}\}}\right)\right)\\
&=
\sum_{y\in B_{NL}}\sum_{i=0}^\infty \sum_{z\in B_{NL}}
E_{z,\omega}\left(\sum_{n=0}^{W_1-1}\mathbbm{1}_{\{y\}}(X_n)\mathbbm{1}_{\{n < T_{B_{NL}}\}}\right)
E_{x,\omega}\left(\mathbbm{1}_{\{z\}}(Y_i)\mathbbm{1}_{\{W_i < T_{B_{NL}}\}}\right)\\
&=
\sum_{z\in \mathcal{Z}}g_Z(x,z,\omega)E_{z,\omega}\left( T_{U_L(z)}\land T_{B_{NL}}\right),
\end{align*} where in the third equality we have used the Markov property for $X$ valid under the probability $P_\omega$ and, in the last one, that $Y$ visits only sites in $\mathcal{Z}$ before the time $T_{B_{NL}}$. 
\end{proof}

Now, to continue with the proof let us define the event 
$$
\mathcal A_2:=\bigcap_{z\in \mathcal{Z}}\left\{ \omega \in \Omega : \left|
E_{z,\omega}(T_{U_L(z)})-E_{0}(T_{U_L})\right|\le
 \epsilon^{-\alpha^*(d)-\delta}
\right\},
$$
where
\begin{equation}
\label{epsilond2}
\alpha^*(d):=3-\alpha(d)=
\begin{cases}
0.5&{\rm if}\ d=3\\
0&{\rm if}\ d\ge 4.
\end{cases}
\end{equation}
By Lemma \ref{step1}, Proposition \ref{green-estimates} and \eqref{convolution2} we have for any $x \in \partial_- B^*_{NL}$ and $\omega \in \mathcal{A}_2$ that
\begin{align*}
E_{x,\omega}(T_{B_{NL}}) &\leq \sum_{z \in \mathcal{Z}} g_Z(x,z,\omega) E_{z,\omega}(T_{U_L(z)}) \\
& \leq \sum_{z \in \mathcal{Z}} g_Z(x,z,\omega) (E_{0}(T_{U_L})+\epsilon^{-\alpha^*(d)-\eta})\\
& \leq E_{x,\omega}(\mathcal{T}^Z_{B_{NL}})(E_{0}(T_{U_L})+\epsilon^{-\alpha^*(d)-\delta})\\
& \leq E_{x,\omega}(\mathcal{T}^Z_{B_{NL}})\left( \frac{1}{\lambda}\E(G_U[\vec{d}\cdot e_1](0)) + \frac{c_6}{\lambda}\epsilon \log L  +\epsilon^{-\alpha^*(d)-\delta}\right)
\end{align*} if $\epsilon,\theta \in (0,1)$ are taken such that $L \geq 2$ and $\epsilon L \leq c_5$. In a similar manner,  since for every $z \in \mathcal{Z}$ we have $T_{U_L(z)} \wedge T_{B_{NL}} = T_{U_L(z)} \wedge T_{\partial_l B_{NL}}$, by using also Proposition \ref{time-truncation} we obtain that 
\begin{align*}
E_{x,\omega}(T_{B_{NL}}) &\geq \sum_{z \in \mathcal{Z}\cap B'_{NL}} g_Z(x,z,\omega)E_{z,\omega}\left(T_{U_L(z)}\wedge T_{B_{NL}}\right)\\
& \geq \sum_{z \in \mathcal{Z}\cap B'_{NL}} g_Z(x,z,\omega)(E_{z,\omega}(T_{U_L(z)})-e^{-c_{15}L})\\
& \geq \sum_{z \in \mathcal{Z}\cap B'_{NL}} g_Z(x,z,\omega)\left(E_{0}(T_{U_L})-\epsilon^{-\alpha^*(d)-\delta}-e^{-c_{15}L}\right)
\\
& \geq E_{x,\omega}(\mathcal{T}^Z_{B'_{NL}})\left(E_{0}(T_{U_L})-\epsilon^{-\alpha^*(d)-\delta}-e^{-c_{15}L}\right)\\
& \geq E_{x,\omega}(\mathcal{T}^Z_{B'_{NL}})\left( \frac{1}{\lambda}\E(G_U[\vec{d}\cdot e_1](0)) - \frac{c_6}{\lambda}\epsilon \log L-\epsilon^{-\alpha^*(d)-\delta}-e^{-c_{15}L}\right)
\end{align*} for any $\omega \in \mathcal{A}_2$ provided that  $\epsilon,\theta \in (0,1)$ are taken such that $L \geq 2$, $\epsilon L \leq c_{16}$ and $\epsilon \in (0,\epsilon_0)$, where $\epsilon_0$ is the one from Proposition \ref{time-truncation}. Next, consider the event
$$
\mathcal A_3:=\left\{ \omega \in \Omega :
\frac{N/2}{2p^+-1}-e^{-c_{22}\epsilon^{-1}}
\le E_{x,\omega}\left(\mathcal T^Z_{B'_{NL}}\right)\le E_{x,\omega}\left(\mathcal T^Z_{B_{NL}}\right)\le
\frac{N/2}{2p_--1}\right\}, 
$$ where $p^\pm$ are those defined in \eqref{p-} and \eqref{p+}, respectively. 
Since $2p^{\pm}-1>0$ by (LD)$_{\eta,\epsilon}$, we see that for $\omega \in \mathcal{A}_2 \cap \mathcal{A}_3$ 
\begin{align*}
E_{x,\omega}(T_{B_{NL}})& \leq \frac{N/2}{2p^--1}\left( \frac{1}{\lambda}\E(G_U[\vec{d}\cdot e_1](0)) + \frac{c_6}{\lambda}\epsilon \log L  +\epsilon^{-\alpha^*(d)-\delta}\right)\\
& \leq \frac{NL/2}{\E(G_U[\vec{d}\cdot e_1](0))-\epsilon^{\alpha(d)-2-\delta}}\left( \frac{1}{\lambda}\E(G_U[\vec{d}\cdot e_1](0)) + \frac{c_6}{\lambda}\epsilon \log L  +\epsilon^{-\alpha^*(d)-\delta}\right)\\
& = \frac{NL}{2}\left( \frac{1}{\lambda}\left(1  + \frac{\epsilon^{\alpha(d)-2-\delta}+c_6\epsilon \log L +\lambda \epsilon^{-\alpha^*(d)-\delta}}{\E(G_U[\vec{d}\cdot e_1](0))-\epsilon^{\alpha(d)-2-\delta}}\right)\right).
\end{align*} Furthermore, if $\epsilon$ is chosen sufficiently small (depending on $\eta,\delta$ and $\theta$) so as to guarantee that $L \geq 2$ together with 
$$
\frac{1}{\theta^2}\cdot \epsilon^{\eta-\delta}< \frac{1}{5}d
$$ then by Proposition \ref{green-estimates} we have  $\E(G_U[\vec{d}\cdot e_1](0)) - \epsilon^{\alpha(d)-2-\delta}\geq \frac{2}{5}d\lambda L^2 - \lambda \epsilon^{-2+\eta-\delta}\geq \frac{1}{5}\lambda L^2$, so that 
 \begin{align*}
 \frac{E_{x,\omega}(T_{B_{NL}})}{NL/2} - \frac{1}{\lambda} &\leq \frac{5}{\lambda^2 L^2}\left(\epsilon^{\alpha(d)-2-\delta} + c_6 \epsilon \log L + \lambda \epsilon^{-\alpha^*(d)-\delta}\right)\\
 & \leq \frac{5}{\lambda^2}\left( \frac{1}{\theta^2}\epsilon^{\alpha(d)-\delta} + 2c_6 \frac{\log L}{L^3} + \frac{1}{2d\theta^2}\epsilon^{-\alpha^*(d)+3-\delta}\right)\\
 & \leq \frac{C(d,\theta)}{\lambda^2}\epsilon^{\alpha(d)-\delta}
 \end{align*} if $\epsilon$ is taken sufficiently small depending on $\delta$, where:
 \begin{enumerate}
 	\item [i.] To obtain the second inequality we have used that 
 	$	\theta \epsilon^{-1} \leq L \leq 2\epsilon^{-1}
 	$ whenever $\epsilon <\theta$ and also that the inequality $\lambda \leq \frac{\epsilon}{2d}$ holds in our case since $\P(\Omega_\epsilon)=1$.
 	\item [ii.] For the third inequality we have used that
 	$L^{-3}\log L \leq \theta^{-3} \epsilon^{3-\delta} \leq \theta^{-3} \epsilon^{\alpha(d)-\delta}
 	$ when $\epsilon$ is sufficiently small so as to guarantee that $\epsilon < \theta$ and $\log L \leq \epsilon^{-\delta}$. 
 \end{enumerate} 
 
 By performing also the analogous computation but for the lower bound instead, we conclude that if $\theta,\epsilon$ are chosen appropriately then for any $\omega \in \mathcal{A}_2 \cap \mathcal{A}_3$ and $x \in \partial_- B^*_{NL}$ we have 
 $$
 \left|\frac{E_{x,\omega}(T_{B_{NL}})}{NL/2} - \frac{1}{\lambda}\right| \leq \frac{c_4}{\lambda^2}\epsilon^{\alpha(d)-\delta}.
 $$ We can now finish the proof by using Propositions
\ref{control2} and \ref{exit-time-z} to obtain an exponential upper
bound of the form $e^{-c_{3}\epsilon^{-\delta}}$ for the probability $\mathbb P(\mathcal A_2^c\cup\mathcal A_3^c)$. 

\section{Proof of Theorem \ref{theorem1} (Part II): the renormalization argument}
\label{parti}

We now finish the proof of Theorem \ref{theorem1} by using the results established in Sections \ref{sec:LLN} and \ref{sers}. To conclude, we only need to show the following proposition.

\begin{proposition}
\label{time-limit} If $d \geq 3$ then for any given $\eta > 0$ and $\delta \in (0,\eta)$ there exists $\epsilon_0=\epsilon_0(d,\eta,\delta) > 0$ such that if (LD)$_{\eta,\epsilon}$ holds for $\epsilon \in (0,\epsilon_0)$ then we have $P_0$-a.s. that
\begin{equation}\label{eq:time-limit}
\liminf_{n\to\infty}\frac{E_{0}\left(T_{n}\right)}{n}\ge 
\frac{1}{\lambda}
+\frac{1}{\lambda^2}O_{d,\eta,\delta}\left(\epsilon^{\alpha(d)-\delta}\right).
\end{equation}
\end{proposition}
Indeed, let us recall from Section \ref{sec:poly} that if our RWRE satisfies (LD)$_{\eta,\epsilon}$ for $\epsilon$ sufficiently small so as to guarantee that $NL \geq M_0$ and $(NL)^{-(15d+5)} \geq e^{-c_2 \epsilon^{-1}}$, where $M_0$ and $c_2$ are respectively the constants from \eqref{eq:defM0} and \eqref{eq:poly}, then the polynomial condition $(P)_{15d+5}$ is satisfied and therefore, \mbox{by Proposition \ref{prop-time},} we have that our RWRE is ballistic with velocity $\vec{v} \in \R^d - \{0\}$ verifying
$$
\lim_{n \rightarrow +\infty} \frac{E_0(T_n)}{n} = \frac{1}{\vec{v} \cdot e_1} > 0.
$$ Together with \eqref{eq:time-limit}, this implies that
$$
\frac{1}{\vec{v} \cdot e_1}\ge
\frac{1}{\lambda}
+\frac{1}{\lambda^2}O_{d,\eta,\delta}\left(\epsilon^{\alpha(d)-\delta}\right).
$$
Taking the reciprocal of this inequality then yields Theorem \ref{theorem1}. Thus, the remainder of the section is devoted to the proof of Proposition \ref{time-limit}.

\subsection{The renormalization scheme}
\label{ren-lb}

The general strategy to prove Proposition \ref{time-limit} will be to apply a renormalization argument similar to the one developed by Berger, Drewitz and Ram\'\i rez in \cite{BDR14} to show that the polynomial condition $(P)_K$ for $K$ sufficiently large implies condition $(T')$ in \cite{Sz02}. We outline the construction of the different scales involved in the argument below.

We start by introducing two sequences $(N_k)_{k \in \N_0}$ and $(N_k')_{k \in \N_0}$ specifying the size of each scale. These sequences will depend on $\epsilon$ and are defined by fixing first
$$
N_0:= NL 
$$ and then for each $k \in \N_0$ setting
$$
N_{k}:= a_k N'_k \hspace{1cm}\text{ and }\hspace{1cm}N'_{k+1}:= b_k N'_k,
$$ where $(a_k)_{k \in \N_0}$ and $(b_k)_{k \in \N_0}$ are two sequences of natural numbers to be chosen appropriately. Observe that, with this definition, for each $k \in \N_0$ we have 
$$
N_{k+1}=\alpha_k N_k
$$ for $\alpha_k:= \frac{a_{k+1}}{a_k}b_k$.
For the renormalization argument to work, we will require $(a_k)_{k \in \N_0}$ and $(b_k)_{k \in \N_0}$ to satisfy the following conditions:
\begin{enumerate}
	\item [C1.] $a_0=2$, i.e. $N'_0 : \frac{NL}{2}$.
	\item [C2.] $(a_k)_{k \in \N_0}$ is increasing.
	\item [C3.] $a_k \leq \frac{1}{22} b_k$ for all $k \in \N_0$, i.e. $N_k \leq \frac{1}{22} N'_{k+1}$ for all $k$.
	\item [C4.] $\sup_{k \in \N_0} \frac{\log \alpha_{k}}{a_k} < +\infty$
	\item [C5.] For each $k \in \N$ one has that
	$$
	\frac{2}{a_{k}} + \frac{1}{12}\cdot \frac{ \log a_{k-1}}{a_{k-1}} + \frac{NL}{\alpha_{k-1}} < \frac{1}{(k+1)^2}.
	$$ 
	\item [C6.] There exists a constant $c_* > 0$ (independent of $k$ and $\epsilon$) such that for all $j \in \N$ 
	$$
	\sum_{i=1}^{j} \log \alpha_{i-1} \leq c_* j^2 \log \epsilon^{-1}.
	$$
	\item [C7.] There exists a constant $c^* > 0$ (independent of $k$ and $\epsilon$) such that 
	$$
	\prod_{k=1}^\infty\left(1 - 8\frac{a_{k-1}}{b_{k-1}}\right) \geq 1 - c^* \epsilon^3.
	$$
\end{enumerate} Notice that, in particular, (C1),(C2) and (C3) together imply that $a_k \leq \alpha_k$ and $\alpha_k \geq 22$ for all $k$. One possible choice of sequences is given for each $k \in \N_0$ by
$$
a_{k+1} := (k+1+K)^3 \hspace{1cm}\text{ and }\hspace{1cm}b_k := a_k (k+1+K)^2,
$$ for $K:=22[\epsilon^{-6}]$. Indeed, (C1), (C2) and (C3) are simple to verify if $\epsilon \in (0,1)$. \mbox{On the other hand,} we have that
$$
\alpha_k = a_{k+1}(k+1+K)^2=(k+1+K)^5
$$ so that (C4) is also satisfied because $\frac{\log (k+1+K)}{k+K} \rightarrow 0$ as $k \rightarrow +\infty$. Moreover, since we have $K \geq 22$ by definition, if $k \in \N$ then 
$$
\frac{2}{a_k} = \frac{1}{(k+K)^2} \cdot \frac{1}{22} < \frac{1}{3} \cdot \frac{1}{(k+1)^2},
$$ 
$$
\frac{1}{12} \cdot \frac{ \log a_{k-1}}{a_{k-1}} \leq \frac{1}{12} \wedge \left( \frac{1}{12} \cdot \frac{1}{(k+K)^2}\cdot \frac{3\log (k+K)}{k+K}\right) \leq \frac{1}{3} \cdot \frac{1}{(k+1)^2}
$$ and 
$$
\frac{NL}{\alpha_{k-1}} \leq \frac{16\epsilon^{-4}}{(k+1+K)^5} \leq \frac{1}{3} \cdot \frac{1}{(k+1)^2}
$$ if $\epsilon$ is sufficiently small so as to guarantee that $\frac{16}{K} \leq \frac{1}{3}$, so that (C5) follows at once. Furthermore, for each $j \in \N$ one has that
$$
\sum_{i=1}^j \log \alpha_{i-1} \leq 5 \sum_{i=1}^j \log (i+K) \leq 5 \sum_{i=1}^j (\log i + \log(K+1)) \leq 5(j^2+j\log(K+1)) \leq 5j^2\log(K+1)
$$ from where (C6) easily follows provided that $\epsilon$ is sufficiently small. Finally, since $\log(1-x) \geq -2x^2$ for $x \leq \frac{1}{2}$, we obtain 
$$
\sum_{k=1}^\infty \log\left(1 - 8 \frac{a_{k-1}}{b_{k-1}}\right) =\sum_{k=1}^\infty \log\left(1 - \frac{8}{(k+K)^2}\right) \geq  \sum_{k=1}^\infty \frac{128}{(k+K)^4} \geq -\frac{128}{22}\cdot \frac{1}{[\epsilon^{-6}]} \cdot \sum_{k=1}^\infty \frac{1}{k^3} \geq -c^* \epsilon^6
$$ for $c^* = \frac{128}{11} \sum_{k=1}^\infty \frac{1}{k^3} $, from which (C7) readily follows.

Next, we introduce the concept of boxes of scale $k \in \N_0$. Given $k \in \N_0$ we say that a set $Q_k \subseteq \Z^d$ is a box of scale $k$ (or simply $k$-\textit{box} to abbreviate) if it is of the form $Q_k=B_{N_k}(x)$ for some $x \in \Z^d$, where for $M \in \N$ the box $B_M(x)$ is defined as in \eqref{beeme}. For any $k$-box $Q_k$ we define its boundaries $\partial_i Q_k$ for $i=+,-,l$ as in Section \eqref{sec:GN}. However, for our current purposes we will need to consider a different definition of its middle-frontal part. Indeed, for any given $k$-box $Q_k = B_{N_k}(x)$ we define its \textit{middle-frontal $k$-part} as
$$
\tilde{Q}_k := \left\{ y \in B_{N_k}(x) : N_k - N'_k \leq (y-x)\cdot e_1 < N_k\,,\,|(y-x)\cdot e_i| < N_k^3 \text{ for }2 \leq i \leq d\right\}
$$ together with its corresponding \textit{back side}
$$
\partial_- \tilde{Q}_k := \left\{ y \in  \tilde{Q}_k : (y-x) \cdot e_1 = N_k - N'_k \right\}.
$$ Observe that for $0$-boxes this definition coincides with the previous one of plain middle-frontal parts. 

For the sequel it will be necessary to introduce for each $k \in \N_0$ the partition \mbox{$\mathcal{C}_{k}=(C_{k}^{(z)})_{z \in \Z^d}$ of $\Z^d$} by middle-frontal $k$-parts defined as
$$
C^{(z)}_k:=\left\{ y \in \Z^d : z_1 N'_k \leq y_1 < (z_1+1) N'_k\,,\,z_i(2N_k^3-1) \leq y_i < (z_i+1)(2N_k^3-1) \text{ for }2 \leq i \leq d\right\}.
$$ Given this partition $\mathcal{C}_k$, for each $x \in \Z^d$ we define
\begin{itemize}
	\item [i.] $z(x)$ as the unique element of $\Z^d$ such that $x \in C^{(z(x))}_k$.
	\item [ii.] $Q_{k}(x)$ as the unique $k$-box having $C^{(z(x))}_{k}$ as its middle-frontal $k$-part.
	\item [iii.] $U_{k}(x)$ as the symmetric slab around $x$ given by
	$$
	U_{k}(x) := \bigcup_{z:|(z-z(x))\cdot e_1| \leq \frac{3}{2}a_{k}-1} C^{(z)}_{k}
	$$ together with its corresponding (inner) boundaries
	$$
	\partial_- U_k(x):= \left\{ y \in U_k(x): y_1 = \left(z(x)-\left(\frac{3}{2}a_{k}-1\right)\right)N'_k \right\}
	$$ and 
	$$
	\partial_+ U_k(x):= \left\{ y \in U_k(x): y_1 =\left(z(x)+\left(\frac{3}{2}a_{k}-1\right)\right)N'_k\right\}.
	$$ Observe that, with this particular choice of boundaries, we have $\partial_- Q_k(x) \subseteq \partial_- U_k(x)$. 
\end{itemize}
Finally, we need to introduce the notion of good and bad $k$-boxes. Given $\omega \in \Omega$, $k \in \N_0$ and $\epsilon > 0$, we will say that: 
\begin{itemize}
	\item [$\bullet$] A $0$-box $Q_0$ is $(\omega,\epsilon)$-\textit{good} if it satisfies the estimates
\begin{equation}
\label{bad-box-0}
\inf_{x\in \tilde Q_0}P_{x,\omega}(X_{T_{{Q_0}}}\in\partial_+Q_0)
\ge 1-e^{-\frac{c_2}{2} \epsilon^{-1}}
\end{equation} and 
\begin{equation}
\label{time-lambda}
\inf_{x\in \partial_-\tilde Q_0}E_{x,\omega}(T_{Q_0})> \left(\frac{1}{\lambda}
-\frac{c_4}{\lambda^2}\epsilon^{\alpha(d)-\delta}\right)N_0',
\end{equation} where $c_2$, $c_4$ are the constants from Theorem \ref{polynomial-satisfied}. Otherwise, we will say that $Q_0$ is $(\omega,\epsilon)$-\textit{bad}. 
\item [$\bullet$] A $(k+1)$-box $Q_{k+1}$ is $(\omega,\epsilon)$-\textit{good} if there exists a $k$-box $Q'_{k}$ such that all $k$-boxes intersecting $Q_{k+1}$ but not $Q'_k$ are necessarily $(\omega,\epsilon)$-good. Otherwise, we will say that $Q_{k+1}$ is $(\omega,\epsilon)$-\textit{bad}.
\end{itemize}

The following lemma, which is a direct consequence of the seed estimates proved in Theorem \ref{polynomial-satisfied}, states that all $0$-boxes are good with overwhelming probability.

\begin{lemma} 
\label{bad0} Given $\eta \in (0,1)$ there exist positive constants $c_{23}$ and $\theta_0$ depending only on $d$ and $\eta$ such that if:
\begin{enumerate}
	\item [i.] The constant $\theta$ from \eqref{defL} is chosen smaller than $\theta_0$,
	\item [ii.] (LD)$_{\eta,\epsilon}$ is satisfied for $\epsilon$ sufficiently small depending only on $d$, $\eta$ and $\theta$,
\end{enumerate}
then for any $0$-box $Q_0$ we have that
$$
\mathbb P(\{ \omega \in \Omega :Q_0 \text{ is $(\omega,\epsilon)$-bad}\})\le e^{-c_{23} N_0^{\frac{\delta}{4}}}.
$$
\end{lemma}
\begin{proof} Notice that, by translation invariance of $\P$, it will suffice to consider the case of $Q_0=B_{NL}$. In this case, \eqref{eq:teo6} implies that the probability of \eqref{time-lambda} not being satisfied is bounded from above by
\begin{equation}
\label{one-step}
e^{-\frac{c_3}{2}N_0^{\frac{\delta}{4}}},
\end{equation} since $N_0^{\frac{\delta}{4}} = L^{\delta} \leq 2^{\delta}\cdot \theta^{\delta} \cdot \epsilon^{-\delta} \leq 2\epsilon^{-\delta}$. 
On the other hand, by Markov's inequality and \eqref{eq:poly} we have
\begin{equation}\label{last-step}
\mathbb P\left(
\sup_{x\in\tilde Q_0}P_{x,\omega}(X_{T_{{Q_0}}}\notin\partial_+Q_0)
> e^{-\frac{c_2}{2}\epsilon^{-1}}\right)
\le e^{\frac{c_2}{2}\epsilon^{-1}}
\sum_{x\in\tilde Q_0}P_{x}(X_{T_{{Q_0}}}\notin\partial_+Q_0) \leq |\tilde{Q}_0| e^{-\frac{c_2}{2}\epsilon^{-1}}.
\end{equation}
Combining \eqref{one-step} with \eqref{last-step} yields the result.
\end{proof} 

Even though the definition of good $k$-box is different for $k \geq 1$, it turns out that such $k$-boxes still satisfy analogues of \eqref{bad-box-0} and \eqref{time-lambda}. The precise estimates are given in Lemmas \ref{poly1} and \ref{poly2} below.

\begin{lemma}
\label{poly1} Given any $\eta \in (0,1)$ there exists $\epsilon_0 > 0$ satisfying that for each $\epsilon \in (0,\epsilon_0)$ there exists a sequence $(d_k)_{k \in \N_0} \subseteq \R_{>0}$ depending on $d, \eta, \delta$ and $\epsilon$ such that for each $k \in \N_0$ the following holds:
\begin{enumerate}
	\item [i.] $d_k \geq \Xi_k d_0$, where $\Xi_k \in (0,1)$ is given by 
	$$
	\Xi_k:= \prod_{j=1}^k \left(1 - \frac{1}{(j+1)^2}\right),
	$$ with the convention that $\prod_{j=1}^0 := 1$.  
	\item [ii.] If $Q_k$ is a $(\omega,\epsilon)$-good $k$-box then 
\begin{equation}
\label{time-good-box0}
\inf_{x\in \tilde Q_k} P_{x,\omega}(X_{T_{{Q_k}}}\in\partial_+Q_k)
\ge 1-e^{-d_k N_k}.
\end{equation}
\end{enumerate}
\end{lemma}

\begin{proof} First, observe that if for $k=0$ we take 
	\begin{equation}\label{eq:defd0}
	d_0:= \frac{c_2}{2\epsilon N_0}
	\end{equation} then condition (i) holds trivially since $\Xi_0 = \frac{1}{2}$ and (ii) also holds by definition of $(\omega,\epsilon)$-good $0$-box. Hence, let us assume that $k \geq 1$ and show that \eqref{time-good-box0} is satisfied for any fixed $(\omega,\epsilon)$-good $k$-box $Q_k$. To this end, for each $x \in \tilde{Q}_k$ we write 
\begin{equation}
\label{consider-box}
P_{x,\omega}(X_{T_{Q_k}}\notin\partial_+ Q_k)
\le 
P_{x,\omega}(X_{T_{Q_k}}\in\partial_l Q_k)+P_{x,\omega}(X_{T_{Q_k}}\in\partial_- Q_k).
\end{equation}
We will show that if $\epsilon$ is sufficiently small (not depending on $k$) and there exists $d_{k-1} > 0$ satisfying that:
\begin{enumerate}
	\item [i'.] $d_{k-1} \geq \Xi_{k-1} d_0$,
	\item [ii'.] For any $(\omega,\epsilon)$-good $(k-1)$-box $Q_{k-1}$ and all $y \in \tilde{Q}_{k-1}$  
$$
\max\{P_{y,\omega}(X_{T_{Q_{k-1}}}\in\partial_l Q_{k-1}),P_{y,\omega}(X_{T_{Q_{k-1}}}\in\partial_- Q_{k-1})\} \leq e^{-d_{k-1}N_{k-1}},
$$
\end{enumerate} then there also exists $d_k > 0$ with $d_k \geq \Xi_k d_0$ such that for all $x \in \tilde{Q}_k$ 
\begin{equation}\label{eq:boun1}
\max\{
P_{x,\omega}(X_{T_{Q_{k}}}\in\partial_l Q_{k}),P_{x,\omega}(X_{T_{Q_{k}}}\in\partial_- Q_{k}) \} \leq \frac{1}{2}e^{-d_k N_k}.
\end{equation} From this, an inductive argument using that (i') and (ii') hold for $d_0$ as in \eqref{eq:defd0} will yield the result. We estimate each term on the left-hand side of \eqref{eq:boun1} separately, starting with the leftmost one.

For this purpose, we recall the partition $\mathcal{C}_{k-1}$ introduced in the beginning of this subsection and define a sequence of stopping times $(\kappa_j)_{j \in \N_0}$ by fixing $\kappa_0:=0$ and then for $j \in \N_0$ setting
$$
\kappa_{j+1}:= \inf \{ n > \kappa_j : X_n \notin Q_{k-1}(X_{\kappa_j})\}.
$$ Having defined the sequence $(\kappa_j)_{j \in \N_0}$ we consider the rescaled random walk $Y=(Y_j)_{j \in \N_0}$ given by the formula
\begin{equation}\label{eq:defY}
Y_j:= X_{\kappa_j \wedge T_{Q_k}}.
\end{equation} Now, since $Q_k$ is $(\omega,\epsilon)$-good, there exists a
$(k-1)$-box $Q'_{k-1}$ such that all $(k-1)$-boxes intersecting $Q_k$ but not $Q'_{k-1}$ are also $(\omega,\epsilon)$-good. Define then $\mathcal B_{Q'_{k-1}}$ as the collection of all $(k-1)$-boxes which intersect $Q'_{k-1}$ and also set $\mathcal{Q}'_{k-1}$ as the smallest horizontal slab $S$ of the form
$$
S = \{ z \in \Z^d : \exists\,\,y \in Q'_{k-1} \text{ with }|(z-y) \cdot e_i| < M \text{ for all }2 \leq i \leq d\}
$$ which contains $\mathcal{B}_{Q'_{k-1}}$. Observe that, in particular, any $(k-1)$-box which does not intersect $\mathcal{Q}'_{k-1}$ is necessarily $(\omega,\epsilon)$-good. Next, we define the stopping times $m_1$, $m_2$ and $m_3$ as follows:
\begin{itemize}
	\item [$\bullet$] $m_1$ is the first time that $Y$ reaches a distance larger than $7N_k^3$ from both $\mathcal Q'_{k-1}$ and $\partial_l Q_k$, the lateral sides  of the box $Q_k$.
	\item [$\bullet$] $m_2$ is the first time that $Y$ exits the box $Q_k$.
	\item [$\bullet$] $m_3:=\inf\{j>m_1:Y_j\in \mathcal{Q}'_{k-1}\}$.
\end{itemize}
Note that on the event $\{X_{T_{Q_k}}\in\partial_l Q_k\}$
we have $P_{x,\omega}$-a.s. $m_1<m_2<+\infty$ so
that the stopping time
$$
m':= m_2 \wedge m_3 - m_1
$$
is well-defined. Furthermore, notice that on the event $\{X_{T_{Q_k}}\in\partial_l Q_k\}$ for each $m_1 < j < m'+m_1$ (such $j$ exist because $m'>1$, see \eqref{eq:cotamprim} below) we have that at time $\kappa_j$ our random walk $X$ is exiting $Q_{k-1}(X_{\kappa_{j-1}})$. This box is necessarily good since it cannot intersect $\mathcal{Q}'_{k-1}$, being $j < m_3$. Moreover, $X$ can exit this box $Q_{k-1}(X_{\kappa_{j-1}})$ either through its back, front or lateral sides. Hence, let us define $n_-$, $n_+$ and $n_l$ as the respective number of such back, frontal and lateral exits, i.e. for $i=-,+,l$ define
$$
n_i:= \# \{ m_1 < j < m'+m_1 : X_{\kappa_j} \in \partial_i Q_{k-1}(X_{\kappa_{j-1}})\}.
$$ 
Furthermore, set $n_+^*$ as the number of pairs of consecutive frontal exits, i.e.
$$
n_+^*:= \#\{ m_1 < j < m'+m_1 -1 :  X_{\kappa_{i}} \in \partial_+ Q_{k-1}(X_{\kappa_{j-1}}) \text{ for }i=j,j+1\}.
$$
Note that with any pair of \textit{consecutive} frontal exits the random walk moves at least a distance $N'_{k-1}$ to the right direction $e_1$, since it must necessarily traverse the entire width of some $C^{(z)}_{k-1}$. Similarly, with any back exit the random walk can move at most a distance $\frac{3}{2}N_{k-1}$ to the left in direction $e_1$, which is the width of any $(k-1)$-box. Therefore, since our starting point $x \in \tilde{Q}_k$ is at a distance not greater than $N'_k$ from $\partial_+ Q_k$, we conclude that on the event $\{ X_{T_{Q_k}} \in \partial_l Q_k\}$ one must have
$$
N'_{k-1} \cdot n_+^* - \frac{3}{2}N_{k-1} \cdot n_- \leq N'_k.
$$ On the other hand, by definition of $m_1$ it follows that
\begin{equation}\label{eq:cotamprim}
m' \geq \frac{7N_k^3}{25N_{k-1}^3} = \frac{7}{25} \alpha_{k-1}^3 =:m''_k.
\end{equation} Furthermore, observe that $n_+ + n_- + n_l = m'-1$ and also that $n_+ - n_+^* \leq n_-+n_l$ since $n_+ - n_+^*$ is the number of frontal exits which were followed by a back or lateral exit. Thus, since $N_{k-1} \geq 2N'_{k-1}$ by assumption, from the above considerations we obtain that
$$
\frac{N'_k}{N'_{k-1}} + 3 \frac{N_{k-1}}{N'_{k-1}} \cdot (n_-+n_l) \geq m' -1. 
$$ From here, a straightforward computation using the definition of $N_j$ and $N_j'$ for $j \geq 0$ shows that 
$$
n_- +n_l \geq \frac{1}{3 a_{k-1}} \cdot (m'-m''_k) + M_k
$$ where
$$
M_k:=\frac{1}{3a_{k-1}}\left(\frac{7}{25} \alpha_{k-1}^3- b_{k-1} -1\right) \geq \frac{1}{15 a_{k-1}}\alpha_{k-1}^3
$$ since $b_{k-1} \leq \alpha_{k-1}$ and $1 \leq \alpha_{k-1} \leq \frac{1}{25} \alpha_{k-1}^3$.
Thus, by conditioning on the value of $m'-m''_k$ it follows that 
\begin{align*}
P_{x,\omega}(X_{T_{Q_k}} \in \partial_l Q_k ) &\leq P_{x,\omega}\left( n_-+n_l \geq \frac{1}{3a_{k-1}} \cdot (m'-m''_k) + M_k\right)\\
& \leq \sum_{N \geq 0} P\left(U_N \geq \frac{1}{3a_{k-1}} \cdot N + M_k\right),
\end{align*} where each $U_N$ is a Binomial random variable of parameters $n:=m''_k+N$ and $p_k:=e^{-d_{k-1}N_{k-1}}$. Using the simple bound $P(U_N \geq r) \leq p^r_k 2^{N+m''_k}$ for $r \geq 0$ yields
$$
P_{x,\omega}(X_{T_{Q_k}} \in \partial_l Q_k ) \leq \left[\frac{1}{1-2p_k^{\frac{1}{3a_{k-1}}}}\right] p_k^{M_k} 2^{m''_k} \leq \left[\frac{1}{1-2p_k^{\frac{1}{3a_{k-1}}}}\right] e^{-d_{k-1}N_{k-1}M_k + m''_k\log 2}.
$$ Now, since 
$$
\inf_{k \in \N_0} \Xi_k = \prod_{j=1}^\infty \left(1 -\frac{1}{(j+1)^2}\right) = \frac{1}{2},
$$ it follows that $d_{k-1}N'_{k-1} \geq \frac{1}{4} d_0 N_0$ because one then has $d_{k-1} \geq \Xi_{k-1}d_0 \geq \frac{1}{2} d_0$ and $N'_{k-1}\geq \frac{1}{2}N_0$. Hence, we obtain that 
\begin{equation} \label{eq:cot1}
p_k^{\frac{1}{3a_{k-1}}}= e^{-\frac{1}{3}d_{k-1}N'_{k-1}} \leq e^{-\frac{1}{12}d_0 N_0} \leq e^{-\frac{c_2 }{24} \epsilon^{-1}} 
\end{equation}
and also 
\begin{align}
-d_{k-1}N_{k-1}M_k + m''_k \log 2 &= -d_{k-1}N_{k}\left( \frac{1}{3a_{k-1}\alpha_{k-1}}M_l - \frac{7}{25}\frac{\log 2}{d_{k-1}N_{k-1}}\alpha_{k-1}^2\right) \nonumber \\
& \leq -d_{k-1}N_{k}\left( \left(\frac{1}{15a_{k-1}}-\frac{28}{25} \frac{\log 2}{ d_{0}N_{0}a_{k-1}}\right)\alpha_{k-1}^2\right) \nonumber\\
& \leq -d_{k-1}N_{k}\left( \left(\frac{1}{15} - \frac{56}{25}\frac{\log 2}{ c_2} \epsilon\right)\frac{\alpha_{k-1}^2}{a_{k-1}}\right) \nonumber\\
&  \leq -d_{k-1}N_{k}\left( \left(\frac{1}{15} - \frac{56}{25}\frac{\log 2}{c_2} \epsilon\right)\alpha_{k-1}\right) \label{eq:cot2}
\end{align} since $a_{k-1} \leq \alpha_{k-1}$. Thus, if $\epsilon$ is taken sufficiently small so as to guarantee that 
$$
\frac{1}{1 - 2e^{-\frac{c_2}{24}\epsilon^{-1}}} \leq 2 \hspace{1cm}\text{ and }\hspace{1cm}
\frac{1}{15} - \frac{56}{25}\frac{\log 2}{c_2} \cdot \epsilon \geq \frac{1}{16}
$$ then, since $\alpha_{k-1} \geq 16$ and $\epsilon \in (0,1)$ by construction, we conclude that   
\begin{align}
P_{x,\omega}(X_{T_{Q_k}} \in \partial_l Q_k ) \leq 2e^{-d_{k-1}N_k} &\leq \frac{1}{2}\exp\left\{ -d_{k-1}N_k + \log 4\right\} \nonumber\\
& \leq 
\frac{1}{2}\exp\left\{ -d_{k-1}N_k\left( 1 - \frac{\log 4}{d_{k-1} N_k}\right)\right\} \nonumber \\
& \leq \frac{1}{2}\exp\left\{ -d_{k-1}N_k\left( 1 - \frac{\log 4}{a_{k}} \cdot \frac{1}{d_{k-1}N'_{k}}\right)\right\} \nonumber\\
& \leq \frac{1}{2}\exp\left\{ -d_{k-1}N_k\left( 1 - \frac{\log 4}{a_{k}} \cdot \frac{4}{d_0N_0}\right)\right\} \nonumber\\
& \leq \frac{1}{2}\exp\left\{ -d_{k-1}N_k\left( 1 - \frac{\log 4}{a_{k}} \cdot \epsilon \cdot \frac{8}{c_2}\right)\right\} \nonumber\\ 
& \leq \frac{1}{2} e^{-\hat{d}_k N_k}, \label{eq:parte1cotafin}
\end{align} for $\hat{d}_k > 0$ given by the formula
$$
\hat{d}_k:= d_{k-1}\left(1 - \frac{1}{a_{k}}\right),
$$ provided that $\epsilon$ is also small enough so as to guarantee that
$$
\frac{8 \log 4}{c_2} \cdot \epsilon < 1.
$$
We turn now to the bound of the remaining term in the left-hand side of \eqref{eq:boun1}. Consider once again the partition $\mathcal{C}_{k-1}$ and notice that if $X_0=x \in \tilde{Q}_k$ then, by construction, we have \mbox{$Q_{k-1}(x) \subseteq U_{k-1}(x)$.} We can then define a sequence $Z=(Z_n)_{n \in \N_0} \subseteq \R$ as follows:
\begin{itemize}
	\item [i.] First, define $\kappa'_0:=0$ and for each $j \in \N$ set
	$$
	\kappa'_j:=\inf \{ n > \kappa'_{j-1} : X_n \in \partial_-U_{k-1}(X_{\kappa'_{j-1}}) \cup \partial_+U_{k-1}(X_{\kappa'_{j-1}})\}. 
	$$
	\item [ii.] Having defined the sequence $(\kappa'_j)_{j \in \N_0}$, for each $j \in \N_0$ define $Z_j:=z(X_{\kappa'_j \wedge T_{Q_k}}) \cdot e_1.$
\end{itemize} The main idea behind the construction of $Z$ is that:
\begin{enumerate}
	\item [$\bullet$] $Z$ starts inside the one-dimensional interval $[l_k,r_k]$, where 
	$$
	l_k = \min \{ z \cdot e_1 : C^{(z)}_{k-1} \cap Q_k \neq \emptyset \} \hspace{1cm}\text{ and } \hspace{1cm} r_k = \max \{ z \cdot e_1 : C^{(z)}_{k-1} \cap Q_k \neq \emptyset \},
	$$ and moves inside this interval until the random walk $X$ first exits $Q_k$. Once this happens, $Z$ remains at its current position forever afterwards.
	\item [$\bullet$] Until $X$ first exits $Q_k$, the increments of $Z$ are symmetric, i.e. $Z_{j+1}-Z_j =\pm \left(  \frac{3}{2}a_{k-1}-1\right)$ for all $j$ with $\kappa'_{j+1} < T_{Q_k}$. 
	
	\item [$\bullet$] Given that $X_{\kappa'_{j}} = y \in Q_k$, if $X$ exits $Q_{k-1}(y)$ through its back side then $X_{\kappa'_{j+1}} \in \partial_- U_{k-1}(y)$, so that $Z_{j+1}-Z_k=-\left(\frac{3}{2}a_{k-1}-1\right)$.
\end{enumerate} Thus, it follows that 
\begin{equation}\label{eq:cotaz}
P_{x,\omega}(X_{T_{Q_k}} \in \partial_- Q_k ) \leq P_{x,\omega}( T^Z_{l_k} < \overline{T}^Z_{r_k})
\end{equation} where $T^Z_{l_k}$ and $\overline{T}^Z_{r_k}$ respectively denote the hitting times for $Z$ of the sets $(-\infty,l_k]$ and $(r_k,+\infty)$. \mbox{To bound the} right-hand side of \eqref{eq:cotaz}, we need to obtain a good control over the jumping probabilities of the random walk $Z$. These will depend on whether the corresponding slab $U_{k-1}$ which $Z$ is exiting at each given time contains a $(\omega,\epsilon)$-bad $(k-1)$-box or not. More precisely, since $Q_k$ is $(\omega,\epsilon)$-good we know that there exists some $(k-1)$-box $\overline{Q}'_{k-1}$ such that all $(k-1)$-boxes which intersect $Q_k$ but not $\overline{Q}'_{k-1}$ are necessarily $(\omega,\epsilon)$-good. Define then 
$$
\left\{\begin{array}{l} L_{k-1}:=\min\{ z \cdot e_1 : C^{(z)}_{k-1} \cap  \overline{Q}_{k-1} \neq \emptyset \} -2a_{k-1}
\\
R_{k-1}:=\max\{ z \cdot e_1 : C^{(z)}_{k-1} \cap  \overline{Q}_{k-1} \neq \emptyset\}+2a_{k-1} \end{array}\right.
$$ and observe that, with this definition, if $y \in \Z^d$ satisfies $y \in C^{(z)}_{k-1}$ for some $z=(z_1,\dots,z_d) \in \Z^d$ with $z_1 \notin [L_{k-1},R_{k-1}]$ then all $(k-1)$-boxes contained in the slab $U_{k-1}(y)$ are necessarily good. From this observation and the uniform ellipticity, it follows that the probability of $Z$ jumping right from a given position $z_1 \in [l_k,r_k]$ is bounded from below by 
$$
p_{k-1}(z_1):=\left\{\begin{array}{ll}(1-e^{-d_{k-1}N_{k-1}})^{\frac{3}{2}a_{k-1} - 1} & \text{ if $z_1 \notin [L_{k-1},R_{k-1}]$}\\
\\ 
\kappa^{\frac{3}{2}N_{k-1}} & \text{ if $z_1 \in [L_{k-1},R_{k-1}].$}\end{array}\right.
$$ Hence, if we write $T^Z_{[L_{k-1},R_{k-1}]}$ to denote the hitting time of $[L_{k-1},R_{k-1}]$ and $\Theta_Z:=\{ T^Z_{l_k} < \overline{T}^Z_{r_k}\}$ then we can decompose
\begin{equation}\label{eq:cot3}
P_{x,\omega}(\Theta_Z) = P_{x,\omega}(\Theta_Z \cap \{T^Z_{[L_{k-1},R_{k-1}]}=+\infty\}) + P_{x,\omega}(\Theta_Z \cap \{T^Z_{[L_{k-1},R_{k-1}]}<+\infty\}).
\end{equation} Now, recall that if $(W_j)_{j \in \N_0}$ is a random walk on $\Z$ starting from $0$ with nearest-neighbor jumps which has probability $p \neq \frac{1}{2}$ of jumping right then, given $a,b \in \N$, the probability $E(-a,b,p)$ of exiting the interval $[-a,b]$ through $-a$ is exactly 
$$
E(-a,b,p)=(1-p)^a \cdot \frac{p^b - (1-p)^b}{p^{a+b}-q^{a+b}} \leq \frac{(1-p)^a}{p^{a+b}-(1-p)^{a+b}}=:\overline{E}(-a,b,p).
$$ Furthermore, if $a,b \in \N$ are such that
$$
\frac{N_k - 2N'_k}{N_{k-1}} \leq a \leq 2 \cdot \frac{N_k}{N_{k-1}} \hspace{1cm}\text{ and }\hspace{1cm} a+b \leq 4 \cdot \frac{N_k}{N_{k-1}}
$$ then for $p'_{k-1}:=(1-e^{-d_{k-1}N_{k-1}})^{\frac{3}{2}a_{k-1}-1}$ and $\epsilon$ sufficiently small (but not depending on $k$) one has 
\begin{equation}\label{eq:cotae}
\overline{E}(-a,b,p'_{k-1}) \leq 2e^{-\tilde{d}_k N_k}.
\end{equation} for 
$$
\tilde{d}_k:=d_{k-1}\left( 1 - \frac{2}{a_k} - \frac{1}{12}\cdot \frac{\log a_{k-1}}{a_{k-1}}\right) \geq d_{k-1}\left(1-\frac{1}{(k+1)^2}\right)> 0.
$$ Indeed, by Bernoulli's inequality which states that $(1-p)^n \geq 1-np$ for all $n \in \N$ and $p \in (0,1)$, for $\epsilon$ sufficiently small so as to guarantee that $\frac{32}{c_2} \cdot \epsilon < \frac{1}{12}$ we have that
\begin{align*}
(1-p'_{k-1})^a & \leq \left(\left(\frac{3}{2}a_{k-1}-1\right)e^{-d_{k-1}N_{k-1}}\right)^a\\
& \leq \exp\left\{-ad_{k-1}N_{k-1} - 2a\log a_{k-1}\right\}\\
& \leq \exp\left\{-d_{k-1}N_k\left( a\frac{N_{k-1}}{N_k} + 2a\frac{\log a_{k-1}}{d_{k-1}N_{k}}\right)\right\}\\
& \leq \exp\left\{-d_{k-1}N_k\left( \frac{N_{k}-2N'_k}{N_k} - 4\frac{\log a_{k-1}}{d_{k-1}N_{k-1}}\right)\right\}\\
& \leq \exp\left\{-d_{k-1}N_k\left( 1 - \frac{2}{a_k} - 16\cdot\frac{\log a_{k-1}}{a_{k-1}}\cdot \frac{1}{ d_0N_0}\right)\right\}\\
& \leq \exp\left\{-d_{k-1}N_k\left( 1 - \frac{2}{a_k} - \frac{32}{c_2}\cdot \epsilon \cdot \frac{\log a_{k-1}}{a_{k-1}}\right)\right\} \leq e^{-\tilde{d}_k N_k}
\end{align*} where we use that $\frac{3}{2}\leq a_{k-1}$ in the second line and $d_{k-1}N'_{k-1}\geq \frac{1}{4}d_0N_0$ in the second-to-last one.  
Similarly, by (C4) we can take $\epsilon$  sufficiently small so as to guarantee that 
$$
\frac{32}{c_2}\cdot \epsilon \cdot \sup_{j \in \N} \left(\frac{\log \alpha_{j-1}}{a_{j-1}}\right) < \frac{1}{2},
$$ in which case we have that
\begin{align}
(p'_{k-1})^{a+b}& \geq 1 - 2a_{k-1}(a+b)e^{-d_{k-1}N_{k-1}} \nonumber\\
& \geq 1 - \exp\left\{ -d_{k-1}N_{k-1}\left(1- \frac{\log 2a_{k-1}(a+b)}{d_{k-1}N_{k-1}} \right)\right\}\nonumber\\
& \geq 1 - \exp\left\{ -d_{k-1}N_{k-1}\left(1- \frac{16}{ c_2}\cdot \epsilon \cdot \left(\frac{\log a_{k-1}}{a_{k-1}} + \frac{\log \alpha_{k-1}}{a_{k-1}} \right)\right)\right\}\nonumber\\
& \geq 1 - \exp\left\{ -d_{k-1}N_{k-1}\left(1- \frac{32}{ c_2}\cdot \epsilon \cdot \frac{\log \alpha_{k-1}}{a_{k-1}}\right)\right\}\nonumber\\
& \geq 1 - \exp\left\{-\frac{ c_2}{16}\epsilon^{-1}\right\}, \label{eq:cotafinal}
\end{align} where we have used that $2 \leq a_{k-1}$ and $4 \leq \alpha_{k-1}$ to obtain the third line. Finally, we have
$$
(1-p'_{k-1})^{a+b} \leq 1 -p'_{k-1} \leq \left(\frac{3}{2}a_{k-1}-1\right)e^{-d_{k-1}N_{k-1}} \leq 2a_{k-1}(a+b)e^{-d_{k-1}N_{k-1}} \leq \exp\left\{ -\frac{ c_2}{16}\epsilon^{-1}\right\}
$$ where, for the last inequality, we have used the bound \eqref{eq:cotafinal}. Hence, by choosing $\epsilon$ sufficiently small (independently of $k$) so as to guarantee that 
$$
(p'_{k-1})^{a+b}-(1-p'_{k-1})^{a+b} \geq \frac{1}{2},
$$ we obtain \eqref{eq:cotae}.

With this, from the considerations made above it follows that
$$
P_{x,\omega}(\Theta_Z \cap \{T^Z_{[L_{k-1},R_{k-1}]}=+\infty\}) \leq \overline{E}(-a,b,p'_{k-1})
$$ for 
$$
a:=\left[\frac{(z(x) \cdot e_1) - l_k}{\frac{3}{2}a_{k-1}-1}\right] \hspace{1cm}\text{ and }\hspace{1cm} b:= \left[\frac{r_k - (z(x)\cdot e_1)}{\frac{3}{2}a_{k-1}-1}\right] +1,
$$ where $[\cdot]$ here denotes the (lower) integer part. Recalling that the width in direction $e_1$ of any $C_{k-1}^{(z)}$ is exactly $N'_{k-1}$ and also that $N'_{j-1} \leq N_{j-1} \leq \frac{1}{8}N'_j$ holds for all $j \in \N$, by using the fact that $x \in \tilde{Q}_k$ it is straightforward to check that
$$
\frac{N_k-N'_k}{N_{k-1}} \leq \frac{\frac{3}{2}N_k - N'_k - N'_{k-1} - (\frac{3}{2}N_{k-1}-N'_{k-1})}{\frac{3}{2}N_{k-1} - N'_{k-1}} \leq a \leq \frac{\frac{3}{2}N_k + N'_{k-1}}{\frac{3}{2}N_{k-1}-N'_{k-1}} \leq \frac{N_{k}+N'_{k}}{N_{k-1}} \leq 2\cdot \frac{N_k}{N_{k-1}}
$$ and 
$$ 
a + b \leq a + \frac{N'_{k}+\frac{3}{2}N_{k-1}}{\frac{3}{2}N_{k-1} - N'_{k-1}} \leq 2 \cdot \frac{N_k+N'_k}{N_{k-1}} \leq 4 \cdot \frac{N_k}{N_{k-1}}
$$ so that \eqref{eq:cotae} in this case yields
\begin{equation}\label{eq:fcota1}
P_{x,\omega}(\Theta_Z \cap \{T^Z_{[L_{k-1},R_{k-1}]}=+\infty\}) \leq  2e^{-\tilde{d}_kN_k}.
\end{equation} To bound the remaining term in the right-hand side of \eqref{eq:cot3}, we separate matters into two cases: either $z(x) \cdot e_1 \leq R_{k-1}$ or $z(x) \cdot e_1 > R_{k-1}$. Observe that if $z(x)\cdot e_1 \leq R_{k-1}$ and we define 
$$
\left\{\begin{array}{l}l(x):= \inf \left\{ j \geq 0 : z(x)\cdot e_1 - j\left(\frac{3}{2}a_{k-1}-1\right) < L_{k-1}\right\}<+\infty\\
\\
z_l(x):= z(x)\cdot e_1 -l(x)\left(\frac{3}{2}a_{k-1}-1\right)\end{array}\right.
$$ then $Z$ necessarily visits the site $z_l(x)$ on the event $\Theta_Z \cap \{ T^Z_{[L_{k-1},R_{k-1}]} < +\infty\}$. On the other hand, if $z(x)\cdot e_1 > R_{k-1}$ and we define  
$$
\left\{\begin{array}{l}r(x):= \sup \left\{ j \geq 0 : z(x)\cdot e_1 - j\left(\frac{3}{2}a_{k-1}-1\right) > R_{k-1}\right\}<+\infty\\ \\ z_r(x):=z(x)\cdot e_1 -r(x)\left(\frac{3}{2}a_{k-1}-1\right)\end{array}\right.
$$ then $Z$ necessarily visits the site $z_r(x)$ on the event $\Theta_Z \cap \{ T^Z_{[L_{k-1},R_{k-1}]} < +\infty\}$. In the first case, by the strong Markov property we can bound 
$$
P_{x,\omega}(\Theta_Z \cap \{T^Z_{[L_{k-1},R_{k-1}]}<+\infty\})  \leq P^Z_{z_l(x),\omega} (T^Z_{l_k} < \overline{T}^Z_{r_k}).
$$ where $P^Z_{z_l(x),\omega}$ denotes the quenched law of $Z$ starting from $z_l(x)$. Using the strong Markov property once again, we can check that
$$
P^Z_{z_l(x),\omega} (T^Z_{l_k} < \overline{T}^Z_{r_k}) \leq \frac{P^Z_{z_l(x),\omega}(D^-)}{P^Z_{z_l(x),\omega}(D^- \cup D^+)} \leq  \frac{P^Z_{z_l(x),\omega}(D^-)}{P^Z_{z_l(x),\omega}(D^+)}
$$ where
$$
D^-:=\{ T^Z_{l_k} < H^Z_{z_l(x)} \} \hspace{1cm} \text{ and }\hspace{1cm}D^+:= \{ \overline{T}^Z_{r_k} < H^Z_{z_l(x)}\}
$$ and we define $H^Z_y:=\inf\{ j > 1 : Z_j = y\}$ for each $y \in \Z$. Now, by forcing $Z$ to always jump right, using that $(r_k - R_{k-1})\cdot N'_{k-1} \leq N'_k$ holds whenever $z(x)\cdot e_1 \leq R_{k-1}$ and also that 
$$
|R_{k-1}-L_{k-1}| \leq \frac{3}{2}a_{k-1} +2 + 4a_{k-1} \leq 8a_{k-1}
$$ we obtain  
$$
P^Z_{z_l(x),\omega}(D^+) \geq \kappa^{8N_{k-1}}\left(p'_{k-1}\right)^{\frac{N'_k}{\frac{3}{2}N_{k-1}-N_{k-1}'}} \geq \kappa^{8	N_{k-1}}\left(p'_{k-1}\right)^{\frac{N_k}N_{k-1}} \geq \frac{1}{2}\kappa^{8N_{k-1}},
$$ where we have used \eqref{eq:cotafinal} to obtain the last inequality. On the other hand, by the Markov property at time $j=1$, we have that
$$
P^Z_{z_l(x),\omega}(T^Z_{l_k} < H^Z_{z_l(x)}) \leq \overline{E}(-a',b',p'_{k-1}) 
$$ for
$$
a':= \left[ \frac{z_l(x) - l_k}{\frac{3}{2}a_{k-1}-1}\right] \hspace{1cm}\text{ and }\hspace{1cm}b':=1.
$$ Using the facts that $x \in \tilde{Q}_k$, $z(x)\cdot e_1 \leq R_{k-1}$, $|R_{k-1}-L_{k-1}| \leq 8a_{k-1}$ and $N'_{k-1} \leq N_{k-1} \leq \frac{1}{22}N_k'$, it is easy to check that
$$
\frac{N_k - N'_k}{N_{k-1}} \leq a' \leq 2 \cdot \frac{N_k}{N_{k-1}} \hspace{1cm}\text{ and }\hspace{1cm} a+b \leq 4 \cdot \frac{N_k}{N_{k-1}},
$$ so that \eqref{eq:cotae} immediately yields 
$$
P_{z_l(x),\omega}^Z(D^-) \leq 2e^{-\tilde{d}_k N_k},
$$ and thus 
\begin{equation}\label{eq:fcota2}
P_{x,\omega}(\Theta_Z \cap \{T^Z_{[L_{k-1},R_{k-1}]}<+\infty\}) \leq 4\kappa^{-8N_{k-1}}e^{-\tilde{d}_k N_k}.
\end{equation} It remains only to treat the case in which $z(x) \cdot e_1 > R_{k-1}$. Recall that in this case we had that $Z$ necessarily visits $z_r(x)$ so that, by the strong Markov property, we have
$$
P_{x,\omega}(\Theta_Z \cap \{T^Z_{[L_{k-1},R_{k-1}]}<+\infty\}) \leq P^Z_{z(x) \cdot e_1,\omega} ( T^Z_{z_r(x)} < \overline{T}^Z_{r_k})  \cdot P_{z_r(x),\omega}^Z(T^Z_{l_k} < \overline{T}^Z_{r_k}).
$$ Notice that, by proceeding as in the previous cases, we obtain
$$
P^Z_{z(x) \cdot e_1,\omega} ( T^Z_{z_r(x)} < \overline{T}^Z_{r_k}) \leq \overline{E}(-a'',b,p'_{k-1})
$$ for 
$$
a'':=\left[\frac{(z(x)\cdot e_1)-z_r(x)}{\frac{3}{2}a_{k-1}-1}\right]\hspace{1cm}\text{ and }\hspace{1cm}b:=\left[\frac{r_k-(z(x)\cdot e_1)}{\frac{3}{2}a_{k-1}-1}\right]+1.
$$ Now, we have two options: either $|l_k - z_r(x)| \leq 11a_{k-1}$ or $|l_k - z_r(x)| > 11a_{k-1}$. In the first case, we have that 
$$
a'' \leq a:=\left[\frac{(z(x)\cdot e_1)-l_{k}}{\frac{3}{2}a_{k-1}-1}\right] \leq a'' + 1 + \frac{11a_{k-1}}{\frac{3}{2}a_{k-1}-1} \leq a''+ 9
$$ so that, by the bound previously obtained on $a$ and $a+b$, we conclude that 
$$
\frac{N_k - 2N'_k}{N_{k-1}} \leq a'' \leq 2 \cdot \frac{N_k}{N_{k-1}} \hspace{1cm}\text{ and }\hspace{1cm} a''+b \leq 4 \cdot \frac{N_k}{N_{k-1}},
$$ which implies that 
\begin{equation}\label{eq:fcota3}
P_{x,\omega}(\Theta_Z \cap \{T^Z_{[L_{k-1},R_{k-1}]}<+\infty\}) \leq P^Z_{z(x) \cdot e_1,\omega} ( T^Z_{z_r(x)} < T^Z_{r_k}) \leq \overline{E}(-a'',b,p'_{k-1}) \leq 2e^{-\tilde{d}_k N_k}.
\end{equation} On the other hand, if $|l_k -z_r(x)| > 11a_{k-1}$ then, since $|z_r(x)-z_l(x)|<11a_{k-1}$ holds because $|R_{k-1}-L_{k-1}|\leq 8a_{k-1}$, the walk $Z$ starting from $z_r(x)$ must necessarily visit $z_l(x)$ if it is to reach $(-\infty,l_k]$ before $(r_k,+\infty)$. Therefore, using the strong Markov property we obtain that 
$$
P_{x,\omega}(\Theta_Z \cap \{T^Z_{[L_{k-1},R_{k-1}]}<+\infty\}) \leq P^Z_{z(x) \cdot e_1,\omega} ( T^Z_{z_r(x)} < \overline{T}^Z_{r_k})  \cdot P_{z_l(x),\omega}^Z(T^Z_{l_k} < \overline{T}^Z_{r_k}).
$$ Since it still holds that $1 \leq a''+b \leq a+b \leq 4\cdot \frac{N_k}{N_{k-1}}$ in this case, then 
$$
(p'_{k-1})^{a''+b} - (1-p'_{k-1})^{a''+b} \geq (p'_{k-1})^{4  \cdot\frac{N_k}{N_{k-1}}} - (1-p'_{k-1}) \geq \frac{1}{2} 
$$ so that
$$
P_{z(x),\omega}^Z(T^Z_{z_r(x)}< T^Z_{r_k}) \leq \overline{E}(-a'',b,p'_{k-1}) \leq 2 (1-p'_{k-1})^{a''}.
$$ On the other hand, as before we have
$$ 
P^Z_{z_l(x),\omega} (T^Z_{l_k} < \overline{T}^Z_{r_k}) \leq  \frac{P^Z_{z_l(x),\omega}(D^-)}{P^Z_{z_l(x),\omega}(D^+)}
$$ but now the distance of $z_l(x)$ from the edges $l_k$ and $r_k$ has changed. Indeed, one now has the bounds
$$
P^Z_{z_l(x),\omega}(D^+) \geq \kappa^{8N_{k-1}}(p'_{k-1})^{\left[\frac{r_k-l_k}{\frac{3}{2}a_{k-1}-1}\right]+1} \geq \kappa^{8N_{k-1}}(p'_{k-1})^{4\cdot \frac{N_k}{N_{k-1}}} \geq \frac{1}{2}\kappa^{8N_{k-1}}
$$ and
$$
P^Z_{z_l(x),\omega}(D^-) \leq \overline{E}(-\hat{a},b',p'_{k-1})
$$ for
$$
\hat{a}:=\left[\frac{z_l(x)-l_k}{\frac{3}{2}a_{k-1}-1}\right]\hspace{1cm}\text{ and }\hspace{1cm}b':=1.
$$ Since clearly $\hat{a}+b' \leq a+b \leq 4\cdot\frac{N_k}{N_{k-1}}$ because $z(x)\cdot e_1 \geq z_l(x)$ by definition, we obtain that 
$$
(p'_{k-1})^{\hat{a}+b'}-(1-p'_{k-1})^{{\hat{a}}+b'} \geq \frac{1}{2},
$$ so that 
$$
P^Z_{z_l(x),\omega}(D^-) \leq 2(1-p'_{k-1})^{\hat{a}}.
$$ We conclude that
$$
P_{x,\omega}(\Theta_Z \cap \{T^Z_{[L_{k-1},R_{k-1}]}<+\infty\}) \leq 8\kappa^{-8N_{k-1}}(1-p'_{k-1})^{a''+\hat{a}}.
$$ Now, recalling that $11a_{k-1}>|z_r(x)-z_l(x)|$, we see that
\begin{align*}
4 \cdot \frac{N_k}{N_{k-1}} \geq a''+\hat{a} &\geq \frac{(z(x)\cdot e_1)- z_r(x) +z_l(x) - l_k - 2\left(\frac{3}{2}a_{k-1}-1\right)}{\frac{3}{2}a_{k-1}-1} \\
& \geq \frac{(z(x)\cdot e_1)-l_k-\left(\frac{3}{2}a_{k-1}-1\right)}{\frac{3}{2}a_{k-1}-1} - \frac{z_r(x)-z_l(x)}{\frac{3}{2}a_{k-1}-1} - 1 \\
& \geq \frac{N_k - N'_k}{N_{k-1}} - 9 \\
& \geq \frac{N_k - 2N'_k}{N_{k-1}}
\end{align*} so that
\begin{equation}\label{eq:fcota4}
P_{x,\omega}(\Theta_Z \cap \{T^Z_{[L_{k-1},R_{k-1}]}<+\infty\}) \leq 8\kappa^{-8N_{k-1}}e^{-\tilde{d}_k N_k}. 
\end{equation} In conclusion, gathering \eqref{eq:fcota1},\eqref{eq:fcota2},\eqref{eq:fcota3} and \eqref{eq:fcota4} yields 
$$
P_{x,\omega}(X_{T_{Q_k}} \in \partial_- Q_k) \leq 10 \kappa^{-8N_{k-1}}e^{-\tilde{d}_kN_k} \leq \frac{1}{2} e^{-d'_k N_k},
$$ where 
$$
d'_{k}:= \tilde{d}_k - 8\log 20\kappa^{-1} \cdot \frac{1}{\alpha_{k-1}}.
$$ Together with \eqref{eq:parte1cotafin}, this gives \eqref{eq:boun1} for $d_k:=\min\{\hat{d}_k,d'_k\}$. It only remains to check that $d_k \geq \Xi_k d_0$. To see this, first notice that (C5) implies that
$$
\hat{d}_k = d_{k-1}\left(1 - \frac{1}{a_k}\right) \geq d_{k-1}\left(1 - \frac{2}{a_k}\right) \geq d_{k-1}\left(1-\frac{1}{(k+1)^2}\right) \geq\Xi_{k}d_0
$$ since $d_{k-1}\geq \Xi_{k-1}d_0$. Thus, it will suffice to check that $d'_k \geq \Xi_k d_0$ holds if $\epsilon$ is sufficiently small. This will follow once again from (C5). Indeed, if $\epsilon$ is such that $\frac{32 \log 20\kappa^{-1}}{c_2}\cdot \epsilon < 1$ then we have that
\begin{align*}
d'_k:&= d_{k-1}\left(1-\frac{2}{a_k}-\frac{1}{2}\cdot \frac{\log a_{k-1}}{a_{k-1}} - \frac{8 \log 20 \kappa^{-1}}{d_{k-1}} \cdot \frac{1}{\alpha_{k-1}}\right)\\
& = d_{k-1}\left(1-\frac{2}{a_k}-\frac{1}{12}\cdot \frac{\log a_{k-1}}{a_{k-1}} - \frac{32 \log 20 \kappa^{-1}}{c_2}\cdot \epsilon \cdot \frac{NL}{\alpha_{k-1}}\right)\\
& \geq d_{k-1} \left(1- \frac{1}{(k+1)^2}\right) \geq \Xi_k d_0. 
\end{align*} This shows that $d_k \geq \Xi_k d_0$ and thus concludes the proof.
\end{proof}

\begin{lemma}\label{lema:18}Given any $\eta \in (0,1)$ and $\delta \in (0,\eta)$ there exists $\epsilon_0 = \epsilon_0(d,\eta,\delta) > 0$ such that if $Q_k$ is a $(\omega,\epsilon)$-good $k$-box for some $\epsilon \in (0,\epsilon_0)$ and $k \in \N_0$ then
		\begin{equation}
		\label{time-good-box01}
		\inf_{x\in \partial_- \tilde Q_k} E_{x,\omega}(T_{Q_k}) > \left(\frac{1}{\lambda} - \frac{c_4}{\lambda} \epsilon^{\alpha(d)-\delta}\right) N'_k \left[\prod_{j=1}^{k}\left( 1 - 8\frac{a_{j-1}}{b_{j-1}}\right)\right]^2
		\end{equation} with the convention that $\prod_{j=1}^0 := 1$. 
\end{lemma}

\begin{proof}
We will prove \eqref{time-good-box01} by induction on $k \in \N_0$. Notice that \eqref{time-good-box01} holds for $k=0$ by definition of $(\omega,\epsilon)$-good $0$-box. Thus, let us assume that $k \geq 1$ and that \eqref{time-good-box01} holds for $(\omega,\epsilon)$-good $(k-1)$-boxes. Consider a $(\omega,\epsilon)$-good $k$-box $Q_k$ and let $x \in \partial_- \tilde{Q}_k$. Observe that if for $j=0,\dots,b_{k-1}$ we define the stopping times
$$
O_{j}:=\inf\{ n \in \N_0 : (X_n - X_0) \cdot e_1 = jN'_{k-1}\} \wedge  T_{Q_k}
$$ then $T_{Q_k}= \sum_{j=1}^{b_{k-1}} O_{j}-O_{j-1}$. Furthermore, if for each $j$ we define $Y_j:=X_{T_{O_j}}$ then it follows from the strong Markov property that
\begin{align}
E_{x,\omega}(T_{Q_k}) &= \sum_{j=1}^{b_{k-1}} E_{x,\omega}(O_j-O_{j-1}) \nonumber\\
&\geq \sum_{j=1}^{b_{k-1}} E_{x,\omega}((O_j-O_{j-1})\mathbbm{1}_{\{O_{j-1}<T_{Q_{k}}\}}) \nonumber\\
& \geq \sum_{j=1}^{b_{k-1}}E_{x,\omega}(E_{Y_{j-1},\omega}(O_1)\mathbbm{1}_{\{O_{j-1}<T_{Q_{k}}\,,\,d(Y_j,\partial_l Q_k) > 25N_{k-1}^3\}})\nonumber\\
& \geq \sum_{j=1}^{b_{k-1}}E_{x,\omega}(E_{Y_{j-1},\omega}(T_{\hat{Q}_{k-1}(Y_{j-1})})\mathbbm{1}_{\{X_{T_{Q'_k}} \in \partial_+ Q'_k\}})\label{eq:cotatfinal},
\end{align} where $d(\cdot,\partial_l Q_k)$ denotes the distance to the lateral side $\partial_l Q_k$ and we define the box $Q'_k$ as 
$$
Q'_k:=\{y \in Q_k : d(y,\partial_l Q_k) > 25N_{k-1}^3\},
$$ together with its frontal side
$$
\partial_+ Q'_k := \partial_+ Q_k \cap Q'_k
$$ and the $(k-1)$-box $\hat{Q}_{k-1}(y)$ for any $y \in \Z^d$ through the formula 
$$
\hat{Q}_{k-1}(y):=B_{N_{k-1}}(y-(N_{k-1}+N'_{k-1})e_1).
$$ 
Observe that if $Y_j \in Q'_{k}$ then $\hat{Q}_{k-1}(Y_j) \subseteq Q_k$ so that $E_{Y_j,\omega}(O_1)\geq E_{Y_j,\omega}(T_{\hat{Q}_{k-1}(Y_j)})$, which explains how we obtained \eqref{eq:cotatfinal}. Now, since there can be at most $|R_{k-1}-L_{k-1}| \leq 8a_{k-1}$ boxes of the form $\hat{Q}_{k-1}(Y_{j-1})$ for $j=1,\dots,b_{k-1}$ which are $(\omega,\epsilon)$-bad, it follows from the inductive hypothesis that
\begin{equation}\label{eq:eqcfinal}
E_{x,\omega}(T_{Q_k}) \geq  \left(\frac{1}{\lambda} - \frac{c_4}{\lambda} \epsilon^{\alpha(d)-\delta}\right) N'_k \left[\prod_{j=1}^{k-1}\left( 1 - 8\frac{a_{j-1}}{b_{j-1}}\right)\right]^2\left( 1 - 8\frac{a_{k-1}}{b_{k-1}}\right)P_{x,\omega}\left(X_{T_{Q'_k}} \in \partial_+ Q'_k\right).
\end{equation} But, by performing a careful inspection of the proof of Lemma \ref{poly1}, one can show that
$$
P_{x,\omega}\left(X_{T_{Q'_k}} \in \partial_+ Q'_k\right) \geq 1 - e^{-\frac{1}{4}d_0N_k}
$$ so that, using that $e^{-x} \leq \frac{1}{x}$ for $x \geq 1$ and also that $N_k \geq N'_k \geq b_{k-1}N'_{k-1} \geq \frac{b_{k-1}}{a_{k-1}}N_{k-1} \geq \frac{b_{k-1}}{a_{k-1}}N_0$, for $\epsilon < c_2$ we obtain 
\begin{align*}
P_{x,\omega}\left(X_{T_{Q'_k}} \in \partial_+ Q'_k\right)& \geq 1 - e^{-\frac{1}{4}d_0 N_k} \\
& \geq 1 - \frac{4}{d_0N_k} \\
& \geq 1- \frac{4}{d_0N_0} \cdot \frac{a_{k-1}}{b_{k-1}}\\
& = 1 - \frac{8}{c_2} \cdot \epsilon \cdot \frac{a_{k-1}}{b_{k-1}} \geq 1 - 8 \frac{a_{k-1}}{b_{k-1}}
\end{align*} which, combined with \eqref{eq:eqcfinal}, yields \eqref{time-good-box01}.
\end{proof}

Finally, we need the following estimate concerning the probability of a $k$-box being $(\omega,\epsilon)$-bad.

\begin{lemma}
\label{poly2} Given $\eta \in (0,1)$ and $\delta \in (0,\eta)$ there exists $\theta_0$ depending only on $d,\eta$ and $\delta$ such that if:
\begin{enumerate}
	\item [i.] The constant $\theta$ from \eqref{defL} is chosen smaller than $\theta_0$,
	\item [ii.] (LD)$_{\eta,\epsilon}$ is satisfied for $\epsilon$ sufficiently small depending only on $d,\eta,\delta$ and $\theta$,
\end{enumerate}
then there exists $c_{24}=c_{24}(d,\eta,\delta,\theta,\epsilon)$ such that for all $k \in \N_0$ and any $k$-box $Q_k$ one has
$$
\mathbb P(\{ \omega \in \Omega :Q_k \text{ is $(\omega,\epsilon)$-bad}\})\le e^{-c_{24}2^k}.
$$
\end{lemma}
\begin{proof} For each $k \in \N_0$ and $\epsilon > 0$ define
$$
q_k(\epsilon):=\P(\{ \omega: Q_k \text{ is }(\omega,\epsilon)\text{-bad}\})
$$ Notice that $q_k$ does not depend on the particular choice of $Q_k$ due to the translation \mbox{invariance of $\P$.} We will show by induction on $k \in \N_0$ that 
\begin{equation}\label{eq:cotapbad}
q_{k} \leq e^{-m_k 2^{k}}
\end{equation} for $m_k$ given by 
$$
m_k:= c_{23}N_0^{\frac{\delta}{4}} - 12d \sum_{j=1}^{k} \frac{\log N_j}{2^j}
$$ with the convention that $\sum_{j=1}^0 := 0$. From \eqref{eq:cotapbad}, the result will follow once we show that $\inf_{k} m_{k} > 0$.

First, observe that \eqref{eq:cotapbad} holds for $k=0$ by Lemma \eqref{bad0}. Therefore, let us assume that $k \geq 1$ and \eqref{eq:cotapbad} holds for $k-1$. Notice that if $Q_k$ is $(\omega,\epsilon)$-bad then necessarily there must be at least two $(\omega,\epsilon)$-bad $(k-1)$-boxes which intersect $Q_k$ but not each other. Since the number of $(k-1)$-boxes which can intersect $Q_k$ is at most 
$$
\frac{3}{2}N_k \cdot \left( 50 N_k^3 \right)^{d-1} \leq (2N_k)^{6d},
$$ then by the union bound and the product structure of $\P$ we conclude that
$$
q_k \leq (2N_k)^{6d} q_{k-1}^2 \leq \exp\left\{ 6d \log 2N_{k} - m_{k-1}2^{k}\right\} \leq e^{-m_k 2^k}.
$$ Thus, it only remains to check that 
\begin{equation}\label{eq:mk0}
\inf_k m_k = c_{23}N_0^{\frac{\delta}{4}} - 12d \sum_{j=1}^{\infty} \frac{\log N_j}{2^j}> 0. 
\end{equation} But notice that by (C6) we have that
$$
\sum_{j=1}^{\infty} \frac{\log N_j}{2^j} \leq \log N_0 + \sum_{j=1}^{\infty} \left(\frac{1}{2^j}\sum_{i=1}^j \log \alpha_{i-1}\right) \leq c \log \epsilon^{-1}
$$ for some constant $c > 0$, from where \eqref{eq:mk0} follows if $\epsilon$ sufficiently small (depending on $\delta$ and $\theta$). 
\end{proof}

Let us now see how to deduce Proposition \ref{time-limit} from
Lemmas \ref{lema:18} and \ref{poly2}. For each $k \in \N_0$ consider the $k$-box given by
$$
Q_k:=\left(-\frac{3}{2}N_k +N'_k, N'_k\right) \times \left(-25N_k^3,25N_k^3\right)^{d-1}.
$$ Using the probability estimate on Lemma \ref{poly2}, the Borel-Cantelli lemma then implies that if $\epsilon,\theta$ are chosen appropriately small then for $\P$-almost every $\omega$ the boxes $Q_k$ are all $(\omega,\epsilon)$-good except for a finite amount of them. In particular, by Lemma \ref{lema:18} we have that for $\P$-almost every $\omega$
$$
\liminf_{k \rightarrow +\infty} \frac{E_{0,\omega}(T_{N'_k})}{N'_k} \geq \left(\frac{1}{\lambda} - \frac{c_4}{\lambda} \epsilon^{\alpha(d)-\delta}\right) \left[\prod_{j=1}^{\infty}\left( 1 - 8\frac{a_{j-1}}{b_{j-1}}\right)\right]^2
$$ By Fatou's lemma, the former implies that 
$$
\liminf_{k \rightarrow +\infty} \frac{E_{0}(T_{N'_k})}{N'_k} \geq \left(\frac{1}{\lambda} - \frac{c_4}{\lambda} \epsilon^{\alpha(d)-\delta}\right) \left[\prod_{j=1}^{\infty}\left( 1 - 8\frac{a_{j-1}}{b_{j-1}}\right)\right]^2
$$ which in turn, since $\lim_{n \rightarrow +\infty} \frac{E_0(T_n)}{n}$ exists by Proposition \ref{prop-time}, yields that 
$$
\liminf_{n \rightarrow +\infty} \frac{E_0(T_n)}{n} \geq \left(\frac{1}{\lambda} - \frac{c_4}{\lambda} \epsilon^{\alpha(d)-\delta}\right) \left[\prod_{j=1}^{\infty}\left( 1 - 8\frac{a_{j-1}}{b_{j-1}}\right)\right]^2.
$$ Recalling now that by (C7) we have
$$
\prod_{j=1}^{\infty}\left( 1 - 8\frac{a_{j-1}}{b_{j-1}}\right) = 1 + O(\epsilon^3),
$$ we conclude the result.


\begin{thebibliography}{}

\bibitem[B12]{B12} N. Berger. {\it Slowdown estimate
for ballistic random walk in random environment}.
J. Eur. Math. Soc., 14(1): 127-174 (2012).

\bibitem[BDR14]{BDR14} N. Berger, A. Drewitz and A.F. Ram\'{\i}rez. {\it
  Effective Polynomial Ballisticity Condition for Randow Walk in Random
  Environment}.  Comm. Pure Appl. Math.  67, 1947-1973 (2014).


\bibitem[BSZ03]{BSZ03}  E. Bolthausen, A.S. Sznitman and O. Zeitoun.
{\it  Cut points and diffusive
random walks in random environment}. Ann. Inst. H. Poincar\'e Probab. Statist. 39
527-555 (2003).

\bibitem[CR16]{CR16} D. Campos and A.F. Ram\'{\i}rez. {\it
	Asymptotic expansion of the invariant measure for ballistic random walks in random environment in the low disorder regime}, to appear in Ann. Probab. arXiv:1511.02945

\bibitem[DR14]{DR14}  A. Drewitz and A.F. Ram\'{\i}rez. {\it
Selected topics in random walks in random environments.}
  Topics in percolative and disordered systems, 23-83, Springer Proc. Math. Stat., 69, Springer, New York, (2014).

\bibitem[K81]{K81} S. Kalikow.
{\it Generalized random walk in a random environment.}
Ann. Probab., 9 753-768 (1981).


\bibitem[L91]{L91} G. F. Lawler. \textit{Intersections of random walks.} 
Probability and its Applications. Birkh\"auser Boston, Inc., Boston, MA. 
(1991). 

\bibitem[R16]{R16} A. F. Ram\'\i rez. \textit{Random walk in the low disorder
ballistic regime}, arXiv:1602.06292

\bibitem[Sa04]{Sa04} C. Sabot.
{\it Ballistic random walks in random environment at low disorder.}
  Ann. Probab.  32,  no. 4, 2996-3023 (2004).

\bibitem[ST16]{ST16} C. Sabot and L. Tournier. {\it Random walks
in Dirichlet environment: an overview}. arXiv:1601.08219

\bibitem[SW69]{SW69} W.L. Smith and W.E. Wilkinson
{\it On branching processes in random environments.}
Ann. Math. Statist. 40,  814--827 (1969).


\bibitem[So75]{So75}
F. Solomon.
{\it Random walks in random environment.}
 Ann. Probab., 3, 1--31 (1975).

\bibitem[Szn00]{Sz00}
A.-S. Sznitman.
 {\it Slowdown estimates and central limit theorem for random
walks in random environment.}
 J. Euro. Math. Soc., 2(2):93--143, (2000).

\bibitem[Szn01]{Sz01}
A.-S. Sznitman.
 {\it On a class of transient random walks in random environment.}
 Ann. Probab., 29(2):724--765, (2001).

\bibitem[Szn02]{Sz02}
A.-S. Sznitman.
{\it An effective criterion for ballistic behavior of random walks in
  random environment.}
 Probab. Theory Related Fields, 122(4):509--544, 2002.

\bibitem[Sz03]{Sz03} A.S. Sznitman. {\it On new examples of ballistic random
  walks in random environment.}  Ann. Probab.  31,  no. 1, 285–322 (2003).


\bibitem[Sz04]{Sz04} A.S. Sznitman.
{\it Topics in random walks in random environment.}
  School and Conference on Probability Theory,  203-266 
ICTP Lect. Notes, XVII, Abdus Salam Int. Cent. Theoret. Phys., Trieste, (2004).

\bibitem[Sz06]{Sz06} A.S. Sznitman.
{\it Random motions in random media.}
  Mathematical statistical physics,  219-242, Elsevier B. V., Amsterdam,
 (2006).

\bibitem[SZ99]{SZ99} A.S. Sznitman and M. Zerner.
{\it A law of large numbers for random walks in random environment. }
Ann. Probab.  27,  no. 4, 1851-1869 (1999).


\bibitem[Z06]{Z06} O. Zeitouni.
{\it Random walks in random environments.}
  J. Phys. A  39,  no. 40, R433-R464 (2006).


\bibitem[ZM01]{ZM01} M. Zerner and F. Merkl.
{\it A zero-one law for planar random walks in random environment.}
  Ann. Probab.  29,  no. 4, 1716-1732 (2001).

\end{thebibliography}
\end{document}